\UseRawInputEncoding
\documentclass[12pt, reqno]{amsart}
\usepackage[margin=1in]{geometry}
\usepackage{amssymb,latexsym,amsmath,amscd,amsfonts}
\usepackage{latexsym}
\usepackage[mathscr]{eucal}
\usepackage{bm}
\usepackage{mathptmx}
\usepackage{amssymb}
\usepackage{amsthm}
\usepackage{xcolor}
\usepackage{dcolumn}
\usepackage[all]{xy}
\usepackage{enumitem}
\usepackage[utf8]{inputenc}

\def\textmatrix#1&#2\\#3&#4\\{\bigl({#1 \atop #3}\ {#2 \atop #4}\bigr)}
\def\dispmatrix#1&#2\\#3&#4\\{\left({#1 \atop #3}\ {#2 \atop #4}\right)}
\newcommand{\beg}{\begin{equation}}
	\newcommand{\eeg}{\end{equation}}
\newcommand{\ben}{\begin{eqnarray*}}
	\newcommand{\een}{\end{eqnarray*}}

\newtheorem{thm}{Theorem}[section]

\newtheorem{lem}[thm]{Lemma}

\newtheorem{prop}[thm]{Proposition}
\numberwithin{equation}{section} \theoremstyle{definition}
\newtheorem{defn}[thm]{Definition}
\newtheorem{rem}[thm]{Remark}

\newtheorem{eg}[thm]{Example}

\newcommand{\C}{\mathbb{C}}
\newcommand{\A}{\mathbb{A}}
\newcommand{\CA}{\overline{\mathbb{A}}}
\newcommand{\D}{\mathbb{D}}
\newcommand{\T}{\mathbb{T}}
\newcommand{\N}{\mathbb{N}}
\newcommand{\Z}{\mathbb{Z}}

\newcommand{\HS}{\mathcal{H}}

\newcommand{\DC}{\overline{\mathbb{D}}}

\def\textmatrix#1&#2\\#3&#4\\{\bigl({#1 \atop #3}\ {#2 \atop #4}\bigr)}
\def\dispmatrix#1&#2\\#3&#4\\{\left({#1 \atop #3}\ {#2 \atop #4}\right)}

\title[Decompositions of operators in an annulus]{JOINT REDUCING SUBSPACES AND ORTHOGONAL DECOMPOSITIONS OF OPERATORS IN AN ANNULUS}

\author[Pal and Tomar]{SOURAV PAL AND NITIN TOMAR}

\address[Sourav Pal]{Mathematics Department, Indian Institute of Technology Bombay,
	Powai, Mumbai - 400076, India.} \email{souravpal@iitb.ac.in , sourav@math.iitb.ac.in}

\address[Nitin Tomar]{Mathematics Department, Indian Institute of Technology Bombay, Powai, Mumbai-400076, India.} \email{tnitin@math.iitb.ac.in}		



\begin{document}

	\maketitle
	
	\begin{abstract}
		A commuting tuple of Hilbert space operators $(T_1, \dotsc, T_n)$ is said to be an \textit{$\mathbb{A}_r^n$-contraction} if the closed polyannulus $\overline{\mathbb A}_r^n$ is a spectral set for it, where $\mathbb A_r=\{ z\in \mathbb C\,:\, r < |z|<1 \}$ for $0<r<1$. We use the term \textit{$\A_r$-contraction} for an $\mathbb A_r^1$-contraction. We provide characterizations for the $\mathbb A_r^n$-unitaries and $\mathbb A_r^n$-isometries and decipher their structures. We find a canonical decomposition for any family of doubly commuting $\mathbb A_r$-contractions and as a special case obtain a Wold type decomposition for doubly commuting $\mathbb A_r$-isometries. We find analogous decompositions for c.n.u. $\mathbb A_r$-contractions. Also, we present a different canonical decomposition for any family of commuting $\mathbb A_r$-contractions and consequently achieve a Wold type decomposition for commuting $\mathbb A_r$-isometries.		
	\end{abstract} 
	
	\section{Introduction}\label{Intro}
	
	
\noindent Throughout the paper, all operators are bounded linear operators acting on complex Hilbert spaces. A contraction is an operator with norm not greater than 1. We denote by $\C, \D$ and $\T$ the complex plane, the unit disk and the unit circle in the complex plane respectively, with center at the origin. For a fixed $r \in (0,1)$, the polyannulus $\A_r^n$ is the cartesian product of $n$ copies of annulus $\A_r$ , where 
	\[
	\A_r:=\left\{z \in \C: r<|z|<1 \right\}.
	\]
	In this paper, we study an operator having $\CA_r$ as a spectral set or more generally an $n$-tuple of commuting operators that has the closed polyannulus $\CA_r^n$ as a spectral set. 
	\begin{defn}
		A commuting tuple of Hilbert space operators $(T_1, \dotsc, T_n)$ is said to be a $\A_r^n$-\textit{contraction} if $\CA_r^n$ is a spectral set for $(T_1, \dotsc, T_n)$, which is to say that the Taylor-joint spectrum $\sigma_T(T_1, \dotsc, T_n) \subseteq \CA_r^n$ and von Neumann's inequality holds, i.e.
		\begin{equation*}
		\|f(T_1, \dotsc, T_n)\| \leq \sup_{z \in \CA_r^n}|f(z)|=\|f\|_{\infty, \CA_r^n}
		\end{equation*}
		for every rational function $f=p\slash q$, where $p,q \in \C[z_1, \dotsc, z_n]$ with $q$ having no zeros in  $\CA_r^n$. An \textit{$\A_r$-contraction} is an $\A_r^{1}$-contraction. 
	\end{defn}
Operator theory on an annulus has a rich literature, e.g. see \cite{Agler, Bello, 	DritschelI, Pas-McCull, comm lifting, N-S1, N-S2, N-S3, Tsikalas}. Most of the one-variable operator theory is about studying a contraction, see \cite{Nagy} and the references therein. A celebrated theorem due to von Neumann \cite{Neumann} states that an operator is a contraction if and only if the closed unit disc $\DC$ is a spectral set for it. This inspiring result motivates one to study operators having other domains (like annulus, multiply connected domain, polyannulus etc.) as spectral sets. An annulus is not a polynomially convex domain unlike a disc. For this reason the spectrum condition, i.e. $\sigma_T(T_1, \dots , T_n) \subseteq \overline{\A}_r^n$ plays a crucial role in the theory of $\A_r^n$-contractions. For example, if $\DC$ is a spectral set for $T$ then $\DC$ is a spectral set for $T^n$ for any positive integer $n$ which is not true if $\overline{\A}_r$ is a spectral set for $T$, because, then the spectrum $\sigma(T^n)$ may not be contained in $\overline{\A}_r$. Unitaries and isometries are special classes of contractions. A unitary is a normal operator with its spectrum lying in $\T$ and an isometry is the restriction of a unitary to an invariant subspace. In an analogous manner, we define unitary and isometry associated with the polyannulus $\A_r^n$. 
\begin{defn}
	Let $T_1, \dotsc, T_n$ be commuting operators acting on a Hilbert space $\mathcal{H}$. Then the tuple $\underline{T}=(T_1, \dotsc, T_n)$ is said to be
	\begin{enumerate}
		\item[(i)] an \textit{$\A_r^n$-unitary} if each $T_j$ is a normal operator and $\sigma_T(\underline{T})$ lies in the distinguished boundary of $\CA_r^n$;
		
		\item[(ii)] an \textit{$\A_r^n$-isometry} if $\sigma_T(\underline{T}) \subseteq \CA_r^n$ and there is an $\A_r^n$-unitary $(U_1, \dotsc, U_n)$ on a Hilbert space $\mathcal{K} \supseteq \HS$ such that $\HS$ is a joint invariant subspace for $U_1, \dots , U_n$ and that $U_j|_\mathcal{H}=T_j$ for $1 \leq j \leq n$.    
	\end{enumerate}
\end{defn}
\noindent The aim of this article is two-fold. First, we find several characterizations for the $\A_r^n$-unitaries and $\A_r^n$-isometries and secondly we determine joint reducing subspaces and orthogonal decompositions for any family of commuting and doubly commuting $\A_r$-contractions. These decompositions are analogous to Wold decomposition of a family of isometries and canonical decomposition of commuting and doubly commuting contractions. Note that a commuting family of operators $\{ T_{\lambda}\,:\, \lambda \in J \}$ acting on a Hilbert space $\HS$ is said to be \textit{doubly commuting} if $T_{\alpha}$ commutes with $T_{\beta}^*$ for every $\alpha , \beta \in J$ with $\alpha \neq \beta$. Also, a closed linear subspace $\HS_1$ of $\HS$ is called a \textit{joint reducing subspace} or \textit{common reducing subspace} for $\{ T_{\lambda}\,:\, \lambda \in J \}$ if $\HS_1$ is an invariant under both $T_{\lambda}$ and $T_{\lambda}^*$ for all $\lambda \in J$.
	\subsection{Motivation}
	 The existence of a common reducing subspace for a family of commuting operators has always been of great interests, e.g. see classic \cite{Nagy} and the references therein. The initial groundbreaking work in this direction was von Neumann-Wold decomposition of an isometry. 
	 
\begin{thm}[von Neumann-Wold, \cite{Wold}]
	Let $V$ be an isometry on a Hilbert space $\mathcal{H}$. Then there is an orthogonal decomposition of $\mathcal{H}$ into an orthogonal sum of two subspaces reducing $V$, say $\mathcal{H}=\mathcal{H}_0 \oplus \mathcal{H}_1$ such that $V|_{\mathcal{H}_{0}}$ is a unitary and $V|_{\mathcal{H}_1}$ is a pure isometry, i.e. a unilateral shift. This decomposition is uniquely determined. 
\end{thm}	
 Indeed, the space $\mathcal{H}_0$ is chosen to be the maximal reducing subspace on which $V$ acts as a unitary. Subsequently, the idea of identifying the maximal unitary part was extended to any contraction, which is famously known as the \textit{canonical decomposition} of a contraction. 	
	\begin{thm}[\cite{Nagy-Foias-001} \& \cite{Langer-1}]\label{thm104}
		For every  contraction $T$ on a Hilbert space $\mathcal{H}$, there is a unique  decomposition of $\mathcal{H}$ into an orthogonal sum of two subspaces reducing $T$, say, $\mathcal{H}=\mathcal{H}_0 \oplus \mathcal{H}_1$ such that $T|_{\mathcal{H}_0}$ is a unitary and $T|_{\mathcal{H}_1}$ is a completely non-unitary (c.n.u.) contraction. Either of the subspaces $\mathcal{H}_0$ or $\mathcal{H}_1$ may coincide with the zero subspace.  Indeed, $\mathcal{H}_0$ is the set of all elements in $\mathcal{H}$ such that 
		\[
		\|T^nh\|=\|h\|=\|T^{*n}h\|, \quad n=1,2, \dotsc \ .
		\]  
	\end{thm}
Note that a contraction $T$ on a Hilbert space $\HS$ is said to be \textit{completely non-unitary} or simply \textit{c.n.u.} if there is no non-zero closed linear subspace of $\HS$ that reduces $T$ and on which $T$ acts a unitary. Naturally, it is rather easier to study the orthogonal parts of a contraction than the operator itself. In \cite{Levan}, Levan went one step ahead and decompose further a c.n.u. contraction into two orthogonal parts.
\begin{thm}[\cite{Levan}, Theorem 1]\label{Levan}
		Let $T$ be a c.n.u. contraction acting on a Hilbert space $\mathcal{H}$. Then $\mathcal{H}$ decomposes into an orthogonal sum $\mathcal{H}=\mathcal{H}_1\oplus \mathcal{H}_2$ such that $\mathcal{H}_1, \mathcal{H}_2$ reduce $T$ and that $T|_{\mathcal{H}_1}$ is a pure isometry and $T|_{\mathcal{H}_2}$ is a completely non-isometry (c.n.i.) contraction.
\end{thm}
We mention here that a \textit{completely non-isometry} or simply a \textit{c.n.i.} contraction is a contraction that does not have an isometry part like a c.n.u. contraction does not have a unitary part. There are various analogues of Wold decomposition for more than one isometry in the literature and one of them is due to S\l oci\'{n}ski for a pair of doubly commuting isometries. 
	\begin{thm}[S\l oci\'{n}ski, \cite{Slocinski}]
		Let $(V_1, V_2)$ be a pair of doubly commuting isometries on a Hilbert space $\mathcal{H}$. Then there is a unique decomposition 
		$
			\mathcal{H}=\mathcal{H}_{uu} \oplus \mathcal{H}_{us} \oplus \mathcal{H}_{su} \oplus  \mathcal{H}_{ss}$ ,	where the spaces $\mathcal{H}_{uu}, \mathcal{H}_{us},\mathcal{H}_{su},\mathcal{H}_{ss}$ are closed linear subspaces of $\mathcal{H}$ reducing $V_1$ and $V_2$
		such that 
		\begin{enumerate}
			\item[(a)] $V_1|_{\mathcal{H}_{uu}}$ and $V_2|_{\mathcal{H}_{uu}}$ are unitaries; 
			\item[(b)] $V_1|_{\mathcal{H}_{us}}$ is a unitary and $V_2|_{\mathcal{H}_{us}}$ is a shift;
			\item[(c)] $V_1|_{\mathcal{H}_{su}}$ is a shift and $V_2|_{\mathcal{H}_{su}}$ is a unitary; 
			\item[(d)] $V_1|_{\mathcal{H}_{ss}}$ and $V_2|_{\mathcal{H}_{ss}}$ are shifts.
		\end{enumerate}
One or more of $\mathcal{H}_{uu}, \mathcal{H}_{us}, \mathcal{H}_{su},\mathcal{H}_{ss}$ may be equal to the trivial space $\{0\}$.
	\end{thm}
\noindent S\l oci\'{n}ski's decomposition was comprehensive in the sense that no further decomposition of the original space was required. Popovici \cite{Popovici} provided an analogue of S\l oci\'{n}ski's result for a pair of commuting isometries, see Theorems 2.8 \& 3.10 in \cite{Popovici}. S\l oci\'{n}ski's work was further generalized in \cite{BurdakII} for any number of doubly commuting isometries with finite dimensional wandering subspaces. Also, there are various interesting orthogonal decompositions for commuting contractions under certain conditions, e.g. see \cite{Albrecht, BCL, mix, BurdakI, BurdakIII,  Cate, Kubrusly, Popovici, Szy}. Eschmeier's work \cite{Esch} on finitely many commuting contractions is a significant contribution to this area. Recently, the first named author of this article developed a new algorithm in \cite{Pal} and found canonical decomposition for any family of doubly commuting contractions. All these interesting works naturally motivate us to study joint reducing subspaces and orthogonal decompositions for operators associated with an annulus.

\subsection{A brief outline of the main results} In Sections \ref{Annulus unitary} \& \ref{annulus isometry} of this article, we study in detail the $\A_r^n$-unitaries and $\A_r^n$-isometries; find characterizations for them and decipher their structures. In particular in Theorem \ref{thm416} we prove that a commuting tuple $(V_1, \dots , V_n)$ is an $\A_r^n$-isometry if and only if each $V_i$ is an $\A_r$-isometry. We show in Section \ref{annulus isometry} that an $\A_r$-isometry satisfies the Brehmer's positivity condition. In Theorem \ref{Wold decomposition}, we show that every $\A_r$-isometry possesses a Wold decomposition that splits it into two orthogonal parts of which one is an $\A_r$-unitary and the other is a pure $\A_r$-isometry. Then in Theorem \ref{canonical decomposition}, we show that every $\A_r$-contraction orthogonally decomposes into an $\A_r$-unitary and a c.n.u. $\A_r$-contraction. A c.n.u. $\A_r$-contraction is an $\A_r$-contraction that does not have any $\A_r$-unitary part. This is parallel to the {canonical decomposition} of a contraction as in Theorem \ref{thm104}. Moving one step ahead, we present a Levan type decomposition for a c.n.u. $\A_r$-contraction in Theorem \ref{Levan type decomposition}. In Section \ref{doubly commuting}, we extend the decomposition results to finitely many doubly commuting $\A_r$-contractions. We also show in Theorem \ref{thm702} that an analogue of canonical decomposition is possible for a commuting tuple of $\mathbb{A}_r$-contractions which is motivated by a similar work due to Burdak \cite{BurdakIII} on commuting contractions. 

\smallskip

In Section \ref{conclusion}, we generalize all these decomposition results to any infinite family of doubly commuting $\A_r$-contractions. To do so, we follow the combinatorial techniques described in \cite{Pal} by the first named author of this paper. Indeed, in Theorem \ref{general case}, we describe a canonical decomposition for any infinite family of doubly commuting $\A_r$-contractions $\mathcal{T}$ acting on a Hilbert space $\HS$ and show that $\HS$ admits an orthogonal decomposition into joint reducing subspaces of $\mathcal{T}$ restricted to which all members of $\mathcal{T}$ are either $\A_r$-unitaries or c.n.u. $\A_r$-contractions. Naturally, as a special case we obtain in Theorem \ref{Wold_infinite} a Wold type decomposition for any number of doubly commuting $\A_r$-isometries. Also, in Theorem \ref{thm805} we determine a Levan type orthogonal decomposition for any infinite family of doubly commuting c.n.u. $\A_r$-contractions which is an analogue of Theorem \ref{Levan}.
Several relevant examples have been discussed throughout the paper.

\section{The Polyannulus unitaries}\label{Annulus unitary}
	
	\vspace{0.2cm}
	
	\noindent Recall that a unitary is a normal operator having its spectrum on the boundary of the unit disk $\D$. In the same spirit, a unitary with respect to the annulus $\mathbb A_r$ is a normal operator with its spectrum lying on the boundary of $\mathbb A_r$. Also, a commuting tuple of $n$ unitaries is a commuting normal $n$-tuple having its Taylor joint spectrum on the $n$-torus $\mathbb T^n$, which is the distinguished boundary of the polydisc $\mathbb D^n$. To go parallel with this, an $\A_r^n$-unitary is naturally defined to be a commuting $n$-tuple of normal operators with its Taylor joint spectrum being a subset of the distinguished boundary $b\A_r^n$ of the polyannulus $\A_r^n$. It is merely mentioned that for any $n \geq 1, b{\mathbb{A}_r^n}=(\partial \mathbb{A}_r)^n$. For the sake of simplicity we shall use the term \textit{$\A_r$-unitary} for an $\A_r^1$-unitary. In this Section, we characterize an $\A_r^n$-unitary and then find a structure theorem for it. We start by recalling the following useful results from the literature \cite{N-S1, N-S2}. 
	
\begin{thm}[\cite{N-S1}, Theorem 3.1]\label{thm302}
	Let $T$ be an invertible operator on a Hilbert space $\mathcal{H}$. Then the following are equivalent: 
	\begin{enumerate}
		\item  $T$ is an $\A_r$-unitary;
		\item $rT^{-1}$ is an $\A_r$-unitary;
		\item there is a unique orthogonal decomposition $\mathcal{H}=\mathcal{H}_u\oplus \mathcal{H}_r$ such that $\mathcal{H}_u$ and $\mathcal{H}_r$ reduce $T$, $T|_{\mathcal{H}_u}$ and $rT^{-1}|_{\mathcal{H}_r}$ are unitaries.  Indeed, we have that
		\begin{equation}\label{eqn301}
			\begin{split}
				\mathcal{H}_u 
				=\overset{\infty}{\underset{n=0}{\bigcap}}Ker(I-T^{*n}T^n)
				\quad \text{and} \quad 
				\mathcal{H}_r 
				=\overset{\infty}{\underset{n=0}{\bigcap}}Ker\left(I-(rT^{-1})^{*n}(rT^{-1})^n\right).
			\end{split}
		\end{equation}
		The spaces $\mathcal{H}_u$ or $\mathcal{H}_r$ may be absent, that is, equal to $\{0\}$.
	\end{enumerate}	
\end{thm}
	
	\begin{prop}[\cite{N-S2}, Lemma 3.4]\label{Polyannulus-each annulus}

A commuting tuple $\underline{T}=(T_1, \dotsc, T_n)$ of operators acting on a Hilbert space $\HS$ is an $\mathbb{A}_r^n$-unitary if and only if $T_1, \dotsc, T_n$ are $\mathbb{A}_r$-unitaries on $\HS$.
	\end{prop}

		\begin{rem}\label{rem303}
		Note that the subspaces $\mathcal{H}_u$ and $\mathcal{H}_r$ as in (\ref{eqn301}) are reducing subspaces for any operator $T'$ that doubly commutes with $T$. Indeed, we have 
		\[
		T'(I-T^{*n}T^n)=(I-T^{*n}T^n)T' \quad \text{for all} \ n \in \N \cup \{0\}.
		\]
		Consequently, $Ker(I-T^{*n}T^n)$ is a closed subspace reducing $T'$ for each $n$ and thus, $\mathcal{H}_u$ also reduces $T'$. For a detailed proof to this, a reader can be referred to Theorem 4.2 in \cite{Pal}.  
	\end{rem}
	  For a normal operator $T$ to be a unitary, it is necessary and sufficient that $I-T^*T=0$. An analogous result holds for an $\mathbb{A}_r$-unitary as shown below.
\begin{prop}\label{prop304}
A normal operator $T$ acting on a Hilbert space $\mathcal{H}$ is an $\mathbb{A}_r$-unitary if and only if 
	\[
	(I-T^*T)(T^*T-r^2I)=0.
	\]
\end{prop}
\begin{proof}
	Let $T$ be an $\mathbb{A}_r$-unitary on $\mathcal{H}$. It follows from Theorem	\ref{thm302} that  $\mathcal{H}$ decomposes into an orthogonal sum $\mathcal{H}=\mathcal{H}_u\oplus \mathcal{H}_r$ such that $\mathcal{H}_u$ and $\mathcal{H}_r$ reduce $T$ and $T|_{\mathcal{H}_u},\, rT^{-1}|_{\mathcal{H}_r}$ are unitaries. With respect to the decomposition $\mathcal{H}=\mathcal{H}_u \oplus \mathcal{H}_r,$ the block matrix of $T$ is given by
	\[
T=	\begin{bmatrix}
		T_1 & 0\\
		0 & rT_2\\
	\end{bmatrix} ,
	\quad \mbox{where} \quad T_1=T|_{\mathcal{H}_u} \quad \text{and} \quad  T_2=r^{-1}T|_{\mathcal{H}_r}.
	\]
	Thus, $T_1$ and $T_2$ are unitaries on $\mathcal{H}_u$ and $\mathcal{H}_r$ respectively. With respect to the decomposition $\mathcal{H}=\mathcal{H}_u \oplus \mathcal{H}_r$, we have
	\begin{equation*}
		\begin{split} 
			(I-T^*T)(T^*T-r^2I)
			&=
			\begin{bmatrix}
				I-T_1^*T_1 & 0\\
				0 & I-r^2T_2^*T_2\\
			\end{bmatrix}
			\begin{bmatrix}
				T_1^*T_1-r^2I & 0\\
				0 & r^2T_2^*T_2-r^2I\\
			\end{bmatrix}\\
			&=
			\begin{bmatrix}
				0 & 0\\
				0 & (1-r^2)I\\
			\end{bmatrix}
			\begin{bmatrix}
				(1-r^2)I & 0\\
				0 & 0\\
			\end{bmatrix}
			=
			\begin{bmatrix}
				0 & 0\\
				0 & 0\\
			\end{bmatrix}.
		\end{split}
	\end{equation*}
	For the converse, let $T$ be a normal operator such that $(I-T^*T)(T^*T-r^2I)=0$ and let $p(z)=(1-|z|^2)(|z|^2-r^2)$. It follows from the spectral mapping theorem that for any $\lambda \in \sigma(T)$, 
	\[
	p(\lambda) \in \sigma(p(T))=\sigma\left((I-T^*T)(T^*T-r^2I)\right)=\{0\}.
	\]
 Consequently, $\sigma(T)$ is contained in the zero set of $p(z)$ which is a subset of $\mathbb{T} \cup r\mathbb{T}$. Thus, $T$ is an $\mathbb{A}_r$-unitary.
\end{proof}

Proposition \ref{Polyannulus-each annulus} tells us that an $\A_r^n$-unitary is an $n$-tuple consisting of commuting $\mathbb{A}_r$-unitaries. Thus, we have the following natural characterization for an $\mathbb A_r^n$-unitary as a consequence of Proposition \ref{prop304}.

\begin{thm}
A commuting tuple of normal operators $(T_1, \dots , T_n)$ is an $\mathbb A_r^n$-unitary if and only if $(I-T_i^*T_i)(T_i^*T_i-r^2I)=0$ for each $i=1, \dots , n$.
\end{thm}	
	\noindent Next, we extend the decomposition result in Theorem \ref{thm302} to the polyannulus case. For an $\mathbb{A}_r^n$-unitary $(T_1, \dots , T_n)$, we shall decompose each $T_i$ at a time and show that at every stage the subspaces are reducing for all $T_1, \dots T_n$. We shall use the following terminologies here. 
	
	\begin{defn} An invertible operator $T$ acting on a Hilbert space $\HS$ is said to be 
		\begin{enumerate}
			\item an \textit{$r$-times unitary} or simply \textit{type} $r$ if there is a unitary $U$ on $\HS$ such that $T=rU$;
			\item a type $u$ if $T$ is a unitary.
		\end{enumerate}
	\end{defn}
	\begin{thm}\label{thm305}
		Let $\underline{T}=(T_1, \dotsc, T_n)$ be a tuple of commuting $\mathbb{A}_r$-unitaries acting on a Hilbert space $\mathcal{H}$. Then there exists a decomposition of $\HS$ into an orthogonal sum of $2^n$ subspaces $\HS_1, \dotsc, \HS_{2^n}$ of $\HS$ such that
		\begin{enumerate}
			\item every $\HS_j
			, 1 \leq j \leq 2^n$, is a joint reducing subspace for $T_1, \dotsc, T_n$ ;
			\item $T_i|_{\HS_j}$ is either a unitary or a $r$-times unitary for $ 1 \leq i \leq n$ and $1 \leq j \leq 2^n$;
			\item if $\Psi_n$ is the collection of all functions $\psi:\{1, \dotsc, n\} \to \{u, r\}$, then corresponding to every $\psi \in \Psi_n$, there is a unique subspace $\HS_\ell (1 \leq \ell \leq 2^n)$ such that the $j$-th component of $(T_1|_{\HS_{\ell}}, \dotsc, T_n|_{\HS_{\ell}})$ is of the type $\psi(j)$ for $1 \leq j \leq n$;
			\item $\HS_1$ is the maximal joint reducing subspace for $T_1, \dotsc, T_n$ such that $T_1|_{\HS_1}, \dotsc, T_n|_{\HS_1}$ are unitaries;
		\item $\HS_{2^n}$ is the maximal joint reducing subspace for $T_1, \dotsc, T_n$ such that $T_1|_{\HS_{2^n}}, \dotsc, T_n|_{\HS_{2^n}}$ are $r$-times unitaries.
		\end{enumerate}
 The decomposition is uniquely determined. One or more members of $\{\HS_j : 1 \leq j \leq 2^n\}$ may be the trivial subspace $\{0\}$.
	\end{thm}
	\begin{proof} For an $\mathbb A_r^n$-unitary, we apply mathematical induction on $n$. The $n=1$ case follows from Theorem \ref{thm302}. We shall present here a proof when $n=2$ to make the algorithm transparent to the readers. Let $T_1, T_2$ be commuting $\mathbb{A}_r$-unitaries on $\mathcal{H}$. From Theorem \ref{thm302}, it follows that there exist orthogonal closed subspaces $\mathcal{H}_u$ and $\mathcal{H}_r$ in $\mathcal{H}$ 
		such that
		$
		\mathcal{H}=\mathcal{H}_u \oplus \mathcal{H}_r.
		$
		The spaces $\mathcal{H}_u, \mathcal{H}_r$ reduce $T_1$ such that $T_1|_{\mathcal{H}_u}$ and $rT^{-1}_1|_{\mathcal{H}_r}$ are unitaries. It follows from Remark \ref{rem303} that $\mathcal{H}_u$ and $\mathcal{H}_r$ reduce $T_2$ as well. Since $T_2|_{\mathcal{H}_u}$ is also an $\mathbb{A}_r$-unitary, Theorem \ref{thm302} gives the decomposition $\mathcal{H}_u=\mathcal{H}_{uu} \oplus \mathcal{H}_{ur}$,
		where $\mathcal{H}_{uu}, \mathcal{H}_{ur}$ are closed subspaces of $\mathcal{H}_u$ reducing $T_2$ such that $T_2|_{\mathcal{H}_{uu}}, rT^{-1}_2|_{\mathcal{H}_{ur}}$ are unitaries. Again, Remark \ref{rem303} yields that  $\mathcal{H}_{uu}$ and $\mathcal{H}_{ur}$ reduce each $T_j|_{\mathcal{H}_u}$ and hence, each $T_j$. Analogously, for the $\mathbb{A}_r$-unitary $T_2|_{\mathcal{H}_r}$, we have the decomposition 		
			$\mathcal{H}_r=\mathcal{H}_{ru} \oplus \mathcal{H}_{rr}$,
		where, $\mathcal{H}_{ru}, \mathcal{H}_{rr}$ are closed subspaces of $\mathcal{H}_r$ reducing $T_2$ such that $T_2|_{\mathcal{H}_{ru}}, rT^{-1}_2|_{\mathcal{H}_{rr}}$ are unitaries and $\mathcal{H}_{ru}, \mathcal{H}_{rr}$ reduce each $T_j$.
		Hence, we have the decomposition 
			$\mathcal{H}=\mathcal{H}_{uu} \oplus \mathcal{H}_{ur} \oplus \mathcal{H}_{ru} \oplus \mathcal{H}_{rr}$,
		where, $\mathcal{H}_{uu},\mathcal{H}_{ur}, \mathcal{H}_{ru}, \mathcal{H}_{rr}$ are closed subspaces of $\mathcal{H}$ reducing each $T_j$ such that 
		\begin{itemize}
			\item[(a)] $T_1|_{\mathcal{H}_{uu}}, \; T_2|_{\mathcal{H}_{uu}}$ are unitaries;
			\item[(b)] $T_1|_{\mathcal{H}_{ur}}$ is a unitary and $T_2|_{\mathcal{H}_{ur}}$ is an $r$-times unitary;
			\item[(c)] $T_1|_{\mathcal{H}_{ru}}$ is an $r$-times unitary and $T_2|_{\mathcal{H}_{ru}}$ is a unitary;
			\item[(d)] $T_1|_{\mathcal{H}_{rr}}, \; T_2|_{\mathcal{H}_{rr}}$ are $r$-times unitaries.
		\end{itemize}
		\noindent 
		Thus, the desired conclusion holds for $n=2$. Let the conclusion be true for $n=k$ for any $k \geq 2$. We prove that it is true for any commuting $(k+1)$-tuple $T_1, \dotsc, T_k, T_{k+1}$ of $\A_r$-unitaries on $\mathcal{H}$. Since $T_{k+1}$ is an $\mathbb{A}_r$-unitary that commutes with $T_1, \dotsc, T_k$,  a repeated application of the method as in Remark \ref{rem303} yields that  each subspace $\mathcal{H}_{j}$ reduces $T_{k+1}$. Consider the collection of operator tuples 
		\[
		\{(T_1|_{\mathcal{H}_j}, \dotsc, T_k|_{\mathcal{H}_j}, \; T_{k+1}|_{\mathcal{H}_j}) \ : \  1 \leq j \leq 2^k\}
		\]
		and let
		 $
	(T_1|_{\mathcal{H}_j}, \dotsc, T_k|_{\mathcal{H}_j}, \; T_{k+1}|_{\mathcal{H}_j})
		$
		be an arbitrary tuple from this collection. By induction hypothesis, each $T_i|_{\mathcal{H}_j} (1 \leq i \leq k)$ is either a unitary or an $r$-times unitary. We apply Theorem \ref{thm302} on  each $T_{k+1}|_{\mathcal{H}_j}$ and so, $\mathcal{H}_j$ can be written as orthogonal sum of two reducing subspaces of $T_{k+1}$, say, $\mathcal{H}_{ju}, \mathcal{H}_{jr}$ such that $T_{k+1}|_{\mathcal{H}_{ju}}$ is a unitary and $T_{k+1}|_{\mathcal{H}_{jr}}$ is an $r$-times unitary. Again, by Remark \ref{rem303}, $\mathcal{H}_{ju}, \mathcal{H}_{jr}$ reduce each of $T_1|_{\mathcal{H}_{j}}, \dotsc, T_k|_{\mathcal{H}_{j}}$. Hence, each tuple 
		$
	(T_1|_{\mathcal{H}_j}, \dotsc, T_k|_{\mathcal{H}_j},  T_{k+1}|_{\mathcal{H}_j})
		$
		orthogonally decomposes into two parts. Continuing this way, for every tuple in the collection, we get $2^{k+1}$ orthogonal subspaces and the other parts of the theorem follow from the induction hypothesis. The proof is now complete.
		
			\end{proof}

		\section{The Polyannulus isometries}\label{annulus isometry}
		
		\vspace{0.2cm}
	
	\noindent An isometry is the restriction of a unitary on an invariant subspace. Analogously, an isometry with respect to the annulus $\mathbb A_r$ is the restriction of an $\A_r$-unitary to an invariant subspace provided that its spectrum is contained in $\CA_r$. Similarly, an $\mathbb A_r^n$-isometry is the restriction of an $\A_r^n$-unitary to a joint-invariant subspace of its components along with the fact that the Taylor joint spectrum of it lies in $\CA_r^n$. We mention once gain that we shall use the term $\A_r$-\textit{isometry} for an $\A_r^1$-isometry. In this Section, we first decipher the structure of an $\A_r$-isometry. Then, we show that an $\A_r^n$-isometry is nothing but a commuting tuple of $\A_r$-isometries. We then find an Wold-type decomposition for an $\A_r^n$-isometry.

	\smallskip
	 
	 Note that a tuple of commuting operators $\underline{T}=(T_1, \dotsc, T_n)$ acting on a Hilbert space $\mathcal{H}$ is said to be \textit{subnormal} if there is a Hilbert space $\mathcal{K} \supseteq \mathcal{H}$ and a commuting tuple of normal operators $\underline{N}=(N_1, \dotsc, N_n)$ in $\mathcal{B}(\mathcal{K})$ such that $\mathcal{H}$ is invariant under $N_1, \dotsc, N_n$ and $N_j|_{\HS}=T_j$ for each $j=1, \dots , n$. The tuple $\underline{N}$ is said to be a \textit{normal extension} of $\underline{T}$. It follows from the theory of subnormal operators (e.g. see \cite{Athavale, Conway, Lubin}) that every subnormal tuple admits a minimal normal extension to the space 
	\[
	\mathcal{K}=\overline{span}\left\{N_1^{*k_1} \dotsc N_m^{*k_m}h \ : \ h \in \mathcal{H} \ \& \ \ k_1, \dotsc, k_m \in \N \cup \{0\} \right\},
	\]
	and a minimal normal extension is unique up to unitary equivalence. We invoke the following results on subnormal operators to prove our theorems of this Section.	
	\begin{thm}[\cite{Conway}, Chapter II]\label{Conway}
		If $S$ is a subnormal operator and $N$ is the minimal normal extension (m.n.e.) of $S$, then  $\sigma(S)=\sigma(N) \  \cup$ some bounded components of $\mathbb{C}\setminus \sigma(N)$.
	\end{thm}
		\begin{lem}[\cite{AthavaleIII}, Proposition 2]\label{AthavaleIII}
		Let $T_1, \dotsc, T_n$ be a set of commuting operators on a Hilbert space $\mathcal{H}$ such that $T_1^*T_1+\dotsc+T_n^*T_n=I_\mathcal{H}$. Then $(T_1, \dotsc, T_n)$ is a subnormal tuple. 
	\end{lem}
	\begin{lem}[\cite{Lubin}, Corollary 1]\label{Lubin}
	Let $\underline{S}=(S_1, \dotsc, S_n)$ be a subnormal tuple acting on a Hilbert space $\mathcal{H}$ and let $\underline{N}=(N_1, \dotsc, N_n)$ be its minimal normal extension. Then each $N_i$ is unitarily equivalent to the minimal normal extension of $S_i$.
\end{lem}
	\noindent For an isometry $V$, it is well known (see CH-I in \cite{Nagy}) that either $\sigma(V)=\DC$ or $\sigma(V) \subseteq \T$. An analogue of this result for an $\mathbb{A}_r$-isometry, which is stated below, was proved in \cite{Bello}. 
	\begin{thm}[\cite{Bello}, Theorem 2.8]  \label{thm404}
		Let $V$ be an invertible operator that admits an extension to an $\A_r$-unitary. Then $V$ is an $\mathbb{A}_r$-isometry and $\sigma(V)\subseteq \mathbb{T} \cup r\mathbb{T}$ or $\sigma(V)=\overline{\mathbb{A}}_r$.
		
	\end{thm}
		
		It follows from Theorem \ref{thm302} that an invertible operator $T$ is an $\mathbb{A}_r$-unitary if and only if $rT^{-1}$ is an $\mathbb{A}_r$-unitary. Below we show that a similar result holds for an $\mathbb{A}_r$-isometry too.
	\begin{lem}\label{lem406}
		Let $V$ be an invertible operator. Then $V$ is an $\mathbb{A}_r$-isometry if and only if $rV^{-1}$ is an $\mathbb{A}_r$-isometry.
	\end{lem}
	\begin{proof}
		Assume that $V$ is an $\mathbb{A}_r$-isometry on $\mathcal{H}$. Then $\sigma(V) \subseteq \CA_r$ and there exists an $\mathbb{A}_r$-unitary $U$ on some space $\mathcal{K} \supseteq \mathcal{H}$ such that $\HS$ is invariant under $V$ and $V=U|_\mathcal{H}$. It follows from the spectral mapping theorem that 
		\[
		\sigma(rV^{-1})=\{r\lambda^{-1} \in \C \ : \ \lambda \in \sigma(V)\} \subseteq \{r\lambda^{-1} \in \C \ : \ \lambda \in \CA_r\} \subseteq \CA_r.
		\]
		Now, we prove that $V^{-1}=U^{-1}|_{\mathcal{H}}$. To see this, note that for $x \in \mathcal{H}$, we have $U^{-1}Vx=x$. Let $y \in \mathcal{H}$. Then there is a unique $x \in \mathcal{H}$ such that $Vx=y$. Thus,
		\[
			V^{-1}y=x = U^{-1}Vx = U^{-1}y.
		\]
		Therefore, $rV^{-1}=rU^{-1}|_{\mathcal{H}}$. It follows from Theorem \ref{thm302} that $rU^{-1}$ is an $\mathbb{A}_r$-unitary on $\mathcal{K}$ and consequently, $rV^{-1}$ is an $\mathbb{A}_r$-isometry on $\mathcal{H}$. A proof of the converse is similar.
	\end{proof}
	 Much like an $\A_r$-unitary, we now seek for an algebraic characterization for an $\A_r$-isometry. Indeed, if $V$ is an $\mathbb{A}_r$-isometry on a space $\mathcal{H}$, then there is an $\mathbb{A}_r$-unitary $U$ acting on $\mathcal{K} \supseteq \mathcal{H}$ such that $V=U|_{\mathcal{H}}$. It follows from Proposition \ref{prop304} that 
	\begin{equation*}
		\begin{split}
			0=(I-U^*U)(U^*U-r^2)
			=-U^{*2}U^2+(1+r^2)U^*U-r^2I.\\		
		\end{split}
	\end{equation*}
Since $U$ is an extension of $V$, we have 
\[
-V^{*2}V^2+(1+r^2)V^*V-r^2I=P_{\mathcal{H}}(-U^{*2}U^2+(1+r^2)U^*U-r^2I)|_{\mathcal{H}}=0.
\]
Consequently, 
$
V^*V+(rV^{-1})^*rV^{-1}-(1+r^2)I=0,
$ 
as $V$ is an invertible operator. Nevertheless, this statement and its converse were proved by Bello and Yakubovich in \cite{Bello}. We provide an alternative proof here.
	\begin{thm}\label{thm407}
		Let $V$ be an invertible operator on a Hilbert space $\mathcal{H}$. Then $V$ is an $\mathbb{A}_r$-isometry if and only if 
		\[
		-V^{*2}V^2+(1+r^2)V^*V-r^2I=0,
		\]
		or equivalently,
		\[   
		V^*V+(rV^{-1})^*rV^{-1}=(1+r^2)I.
		\]	
	\end{thm}
	\begin{proof}
		The forward part is shown above and thus we prove the converse part only. Assume that $V$ is an invertible operator that satisfies $V^{*2}V^2=(1+r^2)V^*V-r^2I$. Therefore, the invertibility of $V$ yields that $V^*V+(rV^{-1})^*rV^{-1}=(1+r^2)I$. It follows from Lemma \ref{AthavaleIII} that $V$ and $rV^{-1}$ are subnormal operators. It is easy to see that for every $h \in \mathcal{H}$, we have
		\begin{equation}\label{norm_equality}
			-\|V^2h\|^2+(1+r^2)\|Vh\|^2-r^2\|h\|^2=0.
		\end{equation}
Let $N$ be a normal extension of $V$ acting on some space containing $\mathcal{H}$. The minimal normal extension of $V$ is the normal operator $N$ restricted to the reducing subspace 
		\[
		\mathcal{K}_0=\overline{\mbox{span}}\{N^{*j}h \ | \ j \geq 0, h \in \mathcal{H}\}.
		\]
We claim that $N$ is an $\mathbb{A}_r$-unitary on $\mathcal{K}_0$. By Proposition \ref{prop304}, it suffices to show that 
		\begin{equation}\label{N_annulus_unitary}
			-N^{*2}N^2+(1+r^2)N^*N-r^2I=0 \quad \mbox{on $\mathcal{K}_0$}.
		\end{equation}
		Since $N$ is a normal operator, it follows from the definition of $\mathcal{K}_0$ that it is sufficient to show :  
		\[
		-N^{*2}N^2+(1+r^2)N^*N-r^2I=0 \quad \text{on} \ \mathcal{H}.
		\]
Let $h \in \mathcal{H}$. Then some routine calculations yield that $\|(-N^{*2}N^2+(1+r^2)N^*N-r^2)h\|^2$ equals
		\begin{equation*}
			\begin{split} 
				&\bigg(\|N^4h\|^2-(1+r^2)\|N^3h\|^2+r^2\|N^2h\|^2\bigg)-(1+r^2)\bigg(\|N^3h\|^2-(1+r^2)\|N^2h\|^2+r^2\|Nh\|^2\bigg)\\
				&+r^2\bigg(\|N^2h\|^2-(1+r^2)\|Nh\|^2+r^2\|h\|^2\bigg),
				\\
			\end{split}
		\end{equation*}
		which is same as
		\begin{equation*}
			\begin{split} 
				&\bigg(\|V^4h\|^2-(1+r^2)\|V^3h\|^2+r^2\|V^2h\|^2\bigg)-(1+r^2)\bigg(\|V^3h\|^2-(1+r^2)\|V^2h\|^2+r^2\|Vh\|^2\bigg)\\
				&+r^2\bigg(\|V^2h\|^2-(1+r^2)\|Vh\|^2+r^2\|h\|^2\bigg) \qquad \qquad \qquad \qquad \qquad \quad \quad \qquad \qquad [\because V=N|_{\mathcal{H}}]
				\\
		&=0,		
			\end{split}
		\end{equation*}
		where the last equality follows from (\ref{norm_equality}). Putting everything together, we have that 
		\[
		\|(-N^{*2}N^2+(1+r^2)N^*N-r^2)h\|=0,
		\]
		for every $h \in \mathcal{H}$ and hence for every $h \in \mathcal{K}_0$. It follows from Proposition \ref{prop304} that $N$ is an $\A_r$-unitary on $\mathcal{K}_0$. Consequently, $V$ is an invertible operator that admits a normal extension to the $\A_r$-unitary $N$. By virtue of Theorem \ref{thm404}, $V$ is an $\A_r$-isometry. The proof is complete. 
	\end{proof}
	It is evident (from the definition) that every $\mathbb{A}_r$-unitary is an $\mathbb{A}_r$-isometry. We intend to find a necessary and sufficient condition for the converse to hold. For this we need the following result due to Taylor \cite{Taylor}.

\begin{thm}[\cite{Taylor}, Theorem 4.9]\label{Taylor's} If $\underline{T}=(T_1, \dotsc, T_n)$ is an $n$-tuple of commuting operators on a Banach space $X$ and if $\sigma_T(\underline{T})=K_1 \cup K_2$, where, $K_1$ and $K_2$ are disjoint compact sets in $\mathbb{C}^n,$ then there are closed subspaces $X_1$ and $X_2$ of $X$ such that
	\begin{enumerate}
		\item $X=X_1\oplus X_2;$
		\item $X_1, X_2$ are invariant under any operator which commutes with each $T_k;$
		\item $\sigma_T(\underline{T}|_{X_1})=K_1$ and $\sigma_T(\underline{T}|_{X_2})=K_2$, where, $\underline{T}|_{X_i}=(T_1|_{X_i}, \dotsc, T_n|_{X_i})$ for $i=1, 2$. 
	\end{enumerate}
\end{thm}

We are going to apply Theorem \ref{Taylor's} to prove the following Proposition. However, a proof to this also follows from Theorem 2.8 of \cite{Bello}.

\begin{prop}\label{prop409}
	An $\A_r$-isometry $V$ is an $\A_r$-unitary if and only if $\sigma(V) \subseteq \mathbb{T} \cup r\mathbb{T}$.
\end{prop}
\begin{proof} 
If $V$ is an $\A_r$-unitary, then it follows trivially that $\sigma(V) \subseteq \T \cup r\T$. Conversely, let $V$ be an $\A_r$-isometry acting on a Hilbert space $\mathcal{H}$  such that $\sigma(V) \subseteq \mathbb{T} \cup r\mathbb{T}$. To prove that $V$ is an $\mathbb{A}_r$-unitary, it suffices to show that $V$ is a normal operator. Let $K_1=\sigma(V) \cap \T$ and $K_2=\sigma(V)\cap r\T$. It follows from Theorem \ref{Taylor's} that there exist closed subspaces $\mathcal{H}_{1}$ and $\mathcal{H}_{2}$ of $\mathcal{H}$ that are invariant under $V$ such that
	$
	\mathcal{H}= \mathcal{H}_{1} \oplus \mathcal{H}_{2}
	$
	and for $V_1=V|_{\mathcal{H}_1}, V_2=V|_{\mathcal{H}_2}$, we have $\sigma(V_1)=K_1 \subseteq \mathbb{T}$ and $\sigma(V_2)=K_2 \subseteq r\mathbb{T}$. Since $\mathcal{H}_1$ and its orthogonal complement $\mathcal{H}_2$ are both invariant under $V$, the spaces $\mathcal{H}_1, \mathcal{H}_2$ reduce $V$.
 	 Observe that $V_1$ and $V_2$ are $\mathbb{A}_r$-isometries on $\mathcal{H}_1$ and $\mathcal{H}_2$ respectively. Let $N_1$ and $N_2$ be the minimal normal extension of $V_1$ and $V_2$ respectively. It follows from Theorem \ref{Conway} that
	 \[
	 \sigma(N_1) \subseteq \sigma(V_1) \subseteq \T \quad \text{and} \quad  \sigma(N_2) \subseteq \sigma(V_2) \subseteq r\T.
	 \] 
	 Consequently, both $N_1$ and $r^{-1}N_2$ are unitaries and so, $V_1$ and $r^{-1}V_2$ are isometries. It follows from the invertibility of $V_1$ and $V_2$ that $V_1$ and $r^{-1}V_2$ are unitaries on $\mathcal{H}_1$ and $\mathcal{H}_2$ respectively. Therefore, $V_1$ and $V_2$ are normal operators and so, $V$ is normal because we can rewrite \[
	 V=\begin{bmatrix}
	 	V_1 & 0 \\
	 	0 & V_2
	 \end{bmatrix}
 \]
   with respect to the decomposition $\mathcal{H}= \mathcal{H}_{1} \oplus \mathcal{H}_{2}$. The proof is now complete. 
\end{proof}
	 Now we produce an example of an $\mathbb{A}_r$-isometry which is not an $\mathbb{A}_r$-unitary. To do this we need to refer a few facts from the literature, e.g. \cite{SarasonI, SarasonII} and the references therein.
	\begin{eg}\label{eg1}
		We got the idea of this example from the seminal work of Sarason \cite{SarasonI}. Consider the space,
		$
		L^2(\partial\mathbb{A}_r)=L^2(\mathbb{T})\oplus L^2(r\mathbb{T}),
		$
		where $L^2(\mathbb{T})$ and $L^2(r\mathbb{T})$ are endowed with normalized Lebesgue measure and the norm on this space is given by
		\begin{equation*}
			\|f\|^2=\frac{1}{2\pi}\overset{2\pi}{\underset{0}{\int}}|f(e^{it})|^2dt+\frac{1}{2\pi}\overset{2\pi}{\underset{0}{\int}}|f(re^{it})|^2dt.    
		\end{equation*}
		Define an operator $S$ on $L^2(\partial\mathbb{A}_r)$ by
		$
		S(f)(z)=zf(z).
		$
		For $f=f_1\oplus f_2 \in L^2(\mathbb{T})\oplus L^2(r\mathbb{T})$, we have 
		\begin{enumerate}
			\item $Sf=g_1 \oplus g_2$ \quad where \quad $g_1(e^{it})=e^{it}f_1(e^{it})$ \quad and  \quad $g_2(re^{it})=re^{it}f_2(re^{it})$;
			
			\vspace{0.1cm}
			
			\item $S^{-1}f=h_1 \oplus h_2$ \quad where \quad $h_1(e^{it})=e^{-it}f_1(e^{it})$ \quad and \quad  $h_2(re^{it})=r^{-1}e^{-it}f_2(re^{it})$;
			
			\vspace{0.1cm}
			
			\item $S^{*}f=k_1 \oplus k_2$ \quad where \quad  $k_1(e^{it})=e^{-it}f_1(e^{it})$ \quad and \quad  $k_2(re^{it})=re^{-it}f_2(re^{it})$.
		\end{enumerate}
	
	\vspace{0.1cm}

\noindent Let $f=f_1\oplus f_2 \in L^2(\mathbb{T})\oplus L^2(r\mathbb{T})$. Then, we have
\begin{equation*}
	\begin{split}
		S^*S(f)&=S^*(g_1 \oplus g_2)   \quad  \quad \quad \text{where} \quad  g_1(e^{it})=e^{it}f_1(e^{it}) \quad \& \quad g_2(re^{it})=re^{it}f_2(re^{it})\\
		&=k_1 \oplus k_2 \qquad \qquad \quad  \text{where} \quad k_1(e^{it})=e^{-it}g_1(e^{it}) \quad \& \quad k_2(re^{it})=re^{-it}g_2(re^{it})\\
		&=k_1 \oplus k_2 \qquad \qquad \quad  \text{where} \quad k_1(e^{it})=f_1(e^{it}) \quad \& \quad k_2(re^{-it})=r^2f_2(re^{it})\\
	\end{split}
\end{equation*}	
and 	
\begin{equation*}
	\begin{split}
		SS^*(f)&=S(k_1 \oplus k_2)   \quad  \quad \quad \text{where} \quad  k_1(e^{it})=e^{-it}f_1(e^{it}) \quad \& \quad k_2(re^{it})=re^{-it}f_2(re^{it})\\
		&=g_1 \oplus g_2 \qquad \qquad \quad  \text{where} \quad g_1(e^{it})=e^{it}k_1(e^{it}) \quad \& \quad g_2(re^{it})=re^{it}k_2(re^{it})\\
		&=g_1 \oplus g_2 \qquad \qquad \quad  \text{where} \quad g_1(e^{it})=f_1(e^{it}) \quad \& \quad g_2(re^{-it})=r^2f_2(re^{it}).\\
	\end{split}
\end{equation*}	
Hence, $SS^*=S^*S$. Moreover, we have
\[
(I-S^*S)(S^*S-r^2I)(f_1\oplus f_2)=(I-S^*S)((1-r^2)f_1 \oplus 0)=0\oplus 0=0,  
\]
for all $f \in L^2(\partial \A_r)$. It follows from Proposition \ref{prop304} that $S$ is an $\A_r$-unitary on $L^2(\partial \A_r)$. Now, we construct an $\mathbb{A}_r$-isometry by restricting $S$ to one of its invariant subspace with certain properties. Let $\mathcal{H}$ be a closed subspace of $L^2(\partial \mathbb{A}_r)$ such that $\mathcal{H}$  is invariant under $S$ and $S^{-1}$ but does not reduce $S$. 
It follows from the Corollary to Theorem 11 in \cite{SarasonI} that such a subspace exists and $\sigma(S|_{\mathcal{H}})=\CA_r$. Hence, $S|_\mathcal{H}$ is an $\mathbb{A}_r$-isometry but not an $\mathbb{A}_r$-unitary by Proposition \ref{prop409}. 

	\end{eg} 

	  The von Neumann-Wold decomposition of an isometry states that any isometry splits into a direct sum of a shift and a unitary. Our next main result is an analogue of the Wold decomposition for an $\A_r$-isometry. Before proving this, we need the following basic and useful results.
	\begin{lem}\label{lem411}
		Given a subnormal contraction $V$ acting on a Hilbert space $\mathcal{H}$, the set
		\[
		\mathcal{H}_0=\{h \in \mathcal{H}: \|V^{*n}h\|=\|h\|=\|V^nh\| \quad  \text{for} \  n=1,2, \dotsc \},
		\]
		is same as the set 
		\[
		\widetilde{\HS}_0=\{h \in \mathcal{H}: \|V^{*n}h\|=\|h\| \quad \text{for} \ n=1,2, \dotsc \}.
		\]
		Furthermore, 
		\[
		\HS_0=\overset{\infty}{\underset{k=0}{\bigcap}} Ker(I-V^kV^{*k}).
		\]
	\end{lem}
	
	\begin{proof}
	
Since $V$ is a subnormal operator, $V^n$ is also a subnormal operator for every positive integer $n$ and consequently, $V^n$ is a hyponormal operator. Thus $\|V^nh\| \geq \|V^{*n}h\|$ for all $n \in \mathbb{N}$ and for every $h \in \mathcal{H}$. For any $h \in \widetilde{\HS}_0$, we have
		$
		\|h\|= \|V^{*n}h\| \leq \|V^nh\| \ \ \mbox{for every $n \in \mathbb{N}$}.
		$
		Since $V$ is a contraction, we have $\|V^nh\| \leq \|h\|$. Thus, $\|V^nh\|=\|h\|$ for every $h \in \widetilde{\HS}_0$ and for all $n \in \mathbb{N}$. Consequently, $h \in \mathcal{H}_0$ and thus $\widetilde{\HS}_0 \subseteq \HS_0$.  The converse that $\HS_0 \subseteq \widetilde{\HS}_0$ is trivial. It is only left to show that $\HS_0=\overset{\infty}{\underset{k=0}{\bigcap}} Ker(I-V^kV^{*k})$. Let $x \in 	Ker(I-V^kV^{*k})$ for all $k \geq 0$. Then $V^kV^{*k}x=x$ and so, $\|V^{*k}x\|=\|x\|$ for all $k \geq 0$. Thus $\overset{\infty}{\underset{k=0}{\bigcap}} Ker(I-V^kV^{*k}) \subseteq \HS_0$. To see the converse, let $x \in \HS_0$ and $k \geq 0$. Then $\|V^{*k}x\|=\|x\|$ and so, 
\[
0=\|x\|^2-\|V^{*k}x\|^2=\langle (I-V^kV^{*k})x, x \rangle=\|(I-V^kV^{*k})^{1\slash 2} x\|^2.
\]
Thus, $(I-V^kV^{*k})x=0$ and consequently $\HS_0 \subseteq \overset{\infty}{\underset{k=0}{\bigcap}} Ker(I-V^kV^{*k})$. The proof is now complete.
	\end{proof}
	\begin{lem}\label{lem412}
		Let $V$ be an invertible operator on a Hilbert space $\mathcal{H}$ and let $\mathcal{L} \subseteq \HS$ be a closed reducing subspace of $V$. Then the following hold: 
		\begin{enumerate}
			\item[(a)] $V|_\mathcal{L}$ and $V^*|_\mathcal{L}$ are invertible operators on $\mathcal{L}$;
			\item[(b)] $\mathcal{L}$ is a reducing subspace for $V^{-1}$.
		\end{enumerate} 
	\end{lem}
	\begin{proof}
		This is obvious as $V$ admits a decomposition $V= V_1 \oplus V_2$ with respect to the orthogonal decomposition $\HS = \mathcal L \oplus (\HS \ominus \mathcal L)$.
	\end{proof}
	
 An isometry $V$ on $\HS$ is called \textit{pure} if it does not have a unitary part in its Wold decomposition, i.e. there is no non-zero closed linear subspace $\HS_1$ of $\HS$ such that $\HS_1$ reduces $V$ and $V|_{\HS_1}$ is a unitary. In an analogous way we define a {pure} $\A_r$-{isometry}.
 
 \begin{defn}
	An $\A_r$-isometry $V$ acting on a Hilbert space $\mathcal{H}$ is said to be 
	a \textit{pure $\mathbb{A}_r$-isometry} if there is no non-zero closed linear subspace $\mathcal H_1$ of $\mathcal{H}$ that reduces $V$ and on which $V$ acts as an $\mathbb{A}_r$-unitary.
	
\end{defn}
 
We now present a main result of this Section.	
	\begin{thm}\label{Wold decomposition} \textit{\textbf{(Wold decomposition for an $\A_r$-isometry)}}
		For every $\mathbb{A}_r$-isometry $V$ on a Hilbert space $\mathcal{H},$ there exists a unique orthogonal decomposition of $\mathcal{H}$ into closed reducing subspaces of $V$, say $\mathcal{H}= \mathcal{H}_{a} \oplus \mathcal{H}_{p},$ such that $V|_{\mathcal{H}_{a}}$ is an $\mathbb{A}_r$-unitary and $V|_{\mathcal{H}_{p}}$ is a pure $\mathbb{A}_r$-isometry; $\mathcal{H}_{a}$ or $\mathcal{H}_{p}$ may equal the trivial subspace $\{0\}$. Moreover, the spaces can be formulated as $\mathcal{H}_a=\mathcal{H}_u \oplus \mathcal{H}_r$ and $\mathcal{H}_p=\mathcal{H}_{p1} \cap \mathcal{H}_{p2}\,$, where
	
		\begin{minipage}[t]{0.5\textwidth}
		\smallskip
		\begin{equation*}
			\begin{split}
				&  \mathcal{H}_u
				=\overset{\infty}{\underset{k=0}{\bigcap}} Ker(I-V^kV^{*k}),\\
				& \mathcal{H}_{p1}  
				= \overset{\infty}{\underset{k=0}{\bigvee}}\ Ran(I-V^kV^{*k}),\\
			\end{split}
		\end{equation*}	
	\end{minipage}	
	\begin{minipage}[t]{0.5\textwidth}
		\smallskip 
		\begin{equation*}
			\begin{split}
				& \mathcal{H}_r 
				=\overset{\infty}{\underset{k=0}{\bigcap}} Ker\left(I-(rV^{-1})^k(rV^{-1})^{*k}\right),\\
				&  \mathcal{H}_{p2} =\overset{\infty}{\underset{k=0}{\bigvee}}\ Ran(I-(rV^{-1})^{k}(rV^{-1})^{*k}).
					\end{split}
		\end{equation*}	
	\end{minipage}
	\end{thm}
	\begin{proof}
	Let us consider the following closed subspaces of $\mathcal{H}$:
		\begin{align*}
			\mathcal{H}_u
			&=\{h \in \mathcal{H}: \|V^nh\|=\|h\|=\|V^{*n}h\|, \ \ n= 0, 1, 2, \dotsc\},\\
			\mathcal{H}_r 
			&=\{h \in \mathcal{H}: \|(rV^{-1})^nh\|=\|h\|=\|(rV^{-1})^{*n}h\|, \ \ n= 0, 1,2,\dotsc\}.
		\end{align*}
		It follows from Lemma \ref{lem406} that $rV^{-1}$ is an $\mathbb{A}_r$-isometry too. Consequently, $V$ and $rV^{-1}$ are contractions and so, Theorem \ref{thm104} yields that $\mathcal{H}_u$ and $\mathcal{H}_r$ reduce $V$ and $rV^{-1}$ respectively to unitaries. It follows from Lemma \ref{lem412} that $\mathcal{H}_r$ also reduces $V$. We now show that $\mathcal{H}_u \cap \mathcal{H}_r=\{0\}$. Let $x \in \mathcal{H}_u \cap \mathcal{H}_r$. Then, $Vx \in \mathcal{H}_r$ and so, $\|rV^{-1}Vx\|=\|Vx\|$. Since $x \in \mathcal{H}_u$, we have $\|Vx\|=\|x\|$. Consequently, 
	\[
	\|x\|=\|Vx\|=\|rV^{-1}Vx\|=r\|x\|\,,
	\]
 which is possible if and only if $x=0$. Let us define 
	$
	\mathcal{H}_a=\mathcal{H}_u \oplus \mathcal{H}_r$ and $\mathcal{H}_p=\mathcal{H}\ominus \mathcal{H}_a$. Let $\mathcal{L} \subseteq \HS$ be closed reducing subspace of $V$ such that $V|_\mathcal{L}$ is an $\A_r$-unitary. It follows from Theorem \ref{thm302} that there is a unique orthogonal decomposition $\mathcal{L}=\mathcal{L}_u\oplus \mathcal{L}_r$ such that $\mathcal{L}_u,\mathcal{L}_r$ reduce $V$ and $V|_{\mathcal{L}_u}, \,rV^{-1}|_{\mathcal{L}_r}$ are unitaries. It follows from the definition of $\mathcal{H}_u$ and $\mathcal{H}_r$ that 
$
\mathcal{L}_u \subseteq \mathcal{H}_u, \, \mathcal{L}_r \subseteq \mathcal{H}_r$ and so $\mathcal{L} \subseteq \mathcal{H}_a$. This shows that $\mathcal{H}_a$ is the largest closed subspace of $\mathcal{H}$ that reduces $V$ to an $\A_r$-unitary. If there is a subspace $\mathcal{L} \subseteq \mathcal{H}_p$ such that $\mathcal{L}$ reduces $V$ to an $\A_r$-unitary, then $\mathcal{L} \subseteq \mathcal{H}_a \cap \mathcal{H}_p=\{0\}$. Hence, $V|_{\mathcal{H}_p}$ is a pure $\A_r$-isometry. Furthermore, Lemma \ref{lem411} yields that 
		\begin{align*}
			\mathcal{H}_u
			&=\{h \in \mathcal{H}: \|V^{*n}h\|=\|h\|,  \ \ n =1,2,\dotsc\},\\
						\mathcal{H}_r 
			&=\{h \in \mathcal{H}: \|(rV^{-1})^{*n}h\|=\|h\|, \ \ n =1,2,\dotsc\},
		\end{align*}
		which can be rewritten as 
		\begin{align*}
			\mathcal{H}_u
			=\overset{\infty}{\underset{k=0}{\bigcap}} Ker(I-V^kV^{*k}) \quad \mbox{and} \quad 
			\mathcal{H}_r 
			=\overset{\infty}{\underset{k=0}{\bigcap}} Ker\left(I-(rV^{-1})^k(rV^{-1})^{*k}\right).
		\end{align*}
	Consequently, we have	
		\begin{align*}
			\begin{split}
				\mathcal{H}_u^{\perp}=\bigg[ \overset{\infty}{\underset{k=0}{\bigcap}} Ker(I-V^kV^{*k}) \bigg]^{\perp}
				=\overset{\infty}{\underset{k=0}{\bigvee}} Ran(I-V^kV^{*k}),
			\end{split}    
		\end{align*}
		and
		\begin{align*}
			\begin{split}
				\mathcal{H}_r^{\perp}=\bigg[ \overset{\infty}{\underset{k=0}{\bigcap}} Ker\left(I-(rV^{-1})^k(rV^{-1})^{*k}\right)\bigg]^{\perp}
				=\overset{\infty}{\underset{k=0}{\bigvee}} Ran(I-(rV^{-1})^{k}(rV^{-1})^{*k}).
			\end{split}
		\end{align*}
The desired structure of $\mathcal{H}_p$ follows from the fact that $\mathcal{H}_p=\mathcal{H}_a^{\perp}=\mathcal{H}_u^{\perp} \cap \mathcal{H}_r^{\perp}$. To see the uniqueness, let $\mathcal{H}=\mathcal{L}_a\oplus \mathcal{L}_p$ be an orthogonal decomposition into reducing subspaces of $V$ such that $V|_{\mathcal{L}_a}$ is $\A_r$-unitary and $V|_{\mathcal{L}_p}$ is a pure $\A_r$-isometry. It follows from the maximality of $\mathcal{H}_a$ that $\mathcal{L}_a \subseteq \mathcal{H}_a$. Hence, $\mathcal{H}_a\ominus \mathcal{L}_a$ is a closed subspace that reduces $V$ to an $\A_r$-unitary. Since 
		$
		\mathcal{H}_a\ominus \mathcal{L}_a \subseteq \mathcal{H}\ominus \mathcal{L}_a = \mathcal{L}_p
		$
	and $V|_{\mathcal{L}_p}$ is a pure $\A_r$-isometry, we have that $\mathcal{H}_a\ominus \mathcal{L}_a=\{0\}$. Therefore, $\mathcal{H}_a= \mathcal{L}_a$ and so, $\mathcal{H}_p=\mathcal{L}_p$. The proof is now complete. 
	\end{proof}
	  We now present an example of a pure $\A_r$-isometry using Example \ref{eg1} and the source of our idea is Sarason's seminal works \cite{SarasonI, SarasonII}. 
\begin{eg}\label{eg2}(\textit{Pure $\mathbb{A}_r$-isometry}) 
	Let $S$ be an operator on $L^2(\partial\mathbb{A}_r)$ (see Example \ref{eg1}) given by
	\[
	S(f)(z)=zf(z) \quad \text{for} \ f \in L^2(\partial\A_r).
	\]
	For every $0 \leq \alpha < 1,$ define the function $w_\alpha$ on $\partial \mathbb{A}_r$ as
	$
		w_\alpha(e^{it})=e^{i\alpha t}$ and $  w_\alpha(re^{it})=r^\alpha e^{i\alpha t},$ for $0 \leq t < 2\pi$. Then, it follows that $Sw_\alpha=w_{\alpha+1}$ and the functions $\{w_{\alpha+n}\}_{n \in \Z}$ are mutually orthogonal in $L^2(\partial \mathbb{A}_r)$. Define 
	\[
	H_{\alpha}^2(\partial \mathbb{A}_r)=\overline{\text{span}}\{w_{\alpha+n} \ : \ n \in \Z\}.
	\] 
	Evidently $H_{\alpha}^2(\partial \mathbb{A}_r)$ is a closed subspace of $L^2(\partial \A_r)$ and is invariant under $S, S^{-1}$ but not $S^*$. Indeed, $S^{-1}w_\alpha=w_{\alpha-1}$. It follows from Theorem \ref{thm404} that $S|_{H^2_\alpha(\partial \mathbb{A}_r)}$ is an $\A_r$-isometry on $H^2_\alpha(\partial \mathbb{A}_r)$. By Theorem \ref{Wold decomposition}, there are closed reducing subspaces $\mathcal{H}_u, \mathcal{H}_r$ and $\mathcal{H}_p$ of $H^2_\alpha(\partial \mathbb{A}_r)$ with
	\[
	H^2_\alpha(\partial \mathbb{A}_r)=\mathcal{H}_u \oplus \mathcal{H}_{r} \oplus \mathcal{H}_p
	\]
	such that $S|_{\mathcal{H}_u}, rS^{-1}|_{\mathcal{H}_{r}}$ are unitaries and $S|_{\mathcal{H}_p}$ is a pure isometry, where, 
	\[
	\mathcal{H}_u=\{f \in H^2_\alpha(\partial \mathbb{A}_r) : \|S^nf\|=\|f\|=\|S^{*n}f\| \quad \text{for} \ \ n= 0, 1,  \dotsc\}
	\]
	and
	\[
	\mathcal{H}_{r}=\{f \in H^2_\alpha(\partial \mathbb{A}_r) : \|(rS^{-1})^nf\|=\|f\|=\|(rS^{-1})^{*n}f\| \quad \text{for} \ \ n= 0, 1,  \dotsc\}.
	\]	
	Let $w_{\alpha+n} \in \mathcal{H}_u$ be arbitrary. Then $\|Sw_{\alpha+n}\|=\|w_{\alpha+n}\|$ and 
	\begin{equation*}
		\begin{split}
			\|w_{\alpha+n}\|^2&=\frac{1}{2\pi}\overset{2\pi}{\underset{0}{\int}}|w_{\alpha+n}(e^{it})|^2dt+\frac{1}{2\pi}\overset{2\pi}{\underset{0}{\int}}|w_{\alpha+n}(re^{it})|^2dt\\
			&=\frac{1}{2\pi}\overset{2\pi}{\underset{0}{\int}}|e^{it(\alpha+n)}|^2dt+\frac{1}{2\pi}\overset{2\pi}{\underset{0}{\int}}|r^{\alpha+n}e^{it(\alpha+n)}|^2dt\\
			&=\frac{1}{2\pi}\overset{2\pi}{\underset{0}{\int}}dt+\frac{1}{2\pi}\overset{2\pi}{\underset{0}{\int}}r^{2(\alpha+n)}dt\\
			&=1+r^{2(\alpha+n)}.\\
		\end{split}
	\end{equation*}
Hence, for every $n \in \Z$, we have
	\[
\|w_{\alpha+n}\|^2-  \|w_{\alpha+n+1}\|^2=r^{2(\alpha+n)}(1-r^2) \quad \text{and} \quad \|w_{\alpha+n}\|^2-  r^2\|w_{\alpha+n-1}\|^2=1-r^2. 
	\]
Take any $f$ in $H_{\alpha}^2(\partial \mathbb{A}_r)$. Then there is some $\{f_n\}_{n \in \Z} \subset \C$ such that $	f=\underset{n \in \Z}{\sum}f_nw_{\alpha+n}$. Consequently,
\[
		\|f\|^2-\|Sf\|^2=\underset{n \in \Z}{\sum}(\|f_nw_{\alpha+n}\|^2-\|f_nw_{\alpha+n+1}\|^2)
		=\underset{n \in \Z}{\sum}|f_n|^2(\|w_{\alpha+n}\|^2-\|w_{\alpha+n+1}\|^2)
\]
and 
\[
	\|f\|^2-\|rS^{-1}f\|^2=\underset{n \in \Z}{\sum}(\|f_nw_{\alpha+n}\|^2-\|rf_nw_{\alpha+n-1}\|^2)
=\underset{n \in \Z}{\sum}|f_n|^2(\|w_{\alpha+n}\|^2-r^2\|w_{\alpha+n-1}\|^2).
\]
Therefore, $\|Sf\|=\|f\|$ or $\|rS^{-1}f\|=\|f\|$ if and only if $f=0$. Hence, $\mathcal{H}_u=\mathcal{H}_r=\{0\}$ and so, $H^2_\alpha(\partial \mathbb{A}_r)=\mathcal{H}_p$. Consequently, $S|_{H^2_{\alpha}(\partial \mathbb{A}_r)}$ is a pure $\mathbb{A}_r$-isometry.
\end{eg}
 A pure isometry $V$ is unitarily equivalent to a unilateral shift with some multiplicity and $\sigma(V)=\overline{\D}$. It was proved in \cite{Bello} that the spectrum of an $\A_r$-isometry is either $\overline{\mathbb A}_r$ or $\mathbb T \cup r\mathbb T$. So, without much surprise it follows from Theorem \ref{Wold decomposition} that the spectrum of a pure $\A_r$-isometry is $\overline{\A}_r$.	
	\begin{prop}\label{prop415}
	The spectrum of a pure $\mathbb{A}_r$-isometry is $\overline{\mathbb{A}}_r$.
\end{prop}
	
Our next theorem is another main  result of this Section.	
\begin{thm}\label{thm416}
	Let $\underline{V}=(V_1, \dotsc, V_n)$ be a tuple of commuting operators acting on a Hilbert space $\mathcal{H}$. Then $\underline{V}$ is an $\A_r^n$-isometry if and only each $V_i$ is an $\A_r$-isometry.
\end{thm}
\begin{proof}
	If $\underline{V}=(V_1, \dotsc, V_n)$ is an $\A_r^n$-isometry, then $(V_1, \dots ,V_n)$ is the restriction of an $\A_r^n$-unitary $(U_1, \dots , U_n)$ to the joint invariant subspace $\HS$ of $U_1, \dots , U_n$. Then each $U_i$ is an $\A_r$-unitary by Proposition \ref{Polyannulus-each annulus}. Also, it follows from the projection property of Taylor joint spectrum that $\sigma(V_i) \subseteq \CA_r$ for $1 \leq i \leq n$. Therefore, it follows that each $V_i$ is an $\A_r$-isometry.
	
	\smallskip
	 Conversely, let us assume that each $V_i$ is an $\A_r$-isometry. Then it follows from the projection property of Taylor-joint spectrum that
$
\sigma_T(\underline{V}) \subseteq \sigma(V_1) \times \dotsc \sigma(V_n) \subseteq \CA_r^n.
$
By Theorem \ref{thm407}, each $V_i$ satisfies the operator identity
$
V_i^*V_i+(rV_i^{-1})^*rV_i^{-1}=(1+r^2)I.
$	
So, we have 
\[
\overset{n}{\underset{i=1}{\sum}}V_i^*V_i+\overset{n}{\underset{i=1}{\sum}}(rV_i^{-1})^*rV_i^{-1}=n(1+r^2)I.
\]
An application of Lemma \ref{AthavaleIII} yields that $\underline{V}$ is a subnormal tuple. Consequently, $\underline{V}=(V_1, \dotsc, V_n)$ admits a normal extension, say $\underline{N}=(N_1, \dotsc, N_n)$ on $\mathcal{K} \supseteq \mathcal{H}$ such that $V_i=N_i|_\mathcal{H}$ for  $i=1, \dotsc, n$. It follows from  Corollary \ref{Lubin} that each $U_i=N_i|_{\mathcal{K}_0}$ for 
\[
\mathcal{K}_0=\overline{\mbox{span}}\{N_1^{*j_1} \dotsc N_n^{*j_n}h \ | \  h \in \mathcal{H} \ \text{\&} \ j_1, \dotsc ,j_n \geq 0 \},
\]
is unitarily equivalent to the minimal normal extension of $V_i$ and hence, an $\mathbb{A}_r$-unitary. Consequently, we get that $\underline{U}=(U_1,  \dotsc, U_n)$ is a  tuple of commuting $\mathbb{A}_r$-unitaries on $\mathcal{K}_0$ such that $V_i=U_i|_\mathcal{H}$. By Proposition \ref{Polyannulus-each annulus}, the tuple $\underline{U}$ is an $\A_r^n$-unitary and so, $\underline{V}$ is an $\mathbb{A}_r^n$-isometry. The proof is complete. 
\end{proof}
	 There are many interesting results in the literature exploring the interplay between operator theory on a given domain and some  positivity conditions, e.g. see the classic \cite{Nagy}. Indeed, several classes of operator tuples are known to exhibit many special properties under certain positivity assumptions. One such important positivity condition is Brehmer's positivity (see \cite{Brehmer} or CH-I in \cite{Nagy}). A tuple of commuting contractions $(T_1, \dotsc, T_n)$ on a Hilbert space $ \mathcal{H}$ is said to satisfy \textit{Brehmer's positivity}  if
	\begin{equation}\label{Szego}
	S(u)=\underset{v \subset u}\sum(-1)^{|v|}(T^{e(v)})^*T^{e(v)} \geq 0
	\end{equation}
	for every subset $u$ of $\{1,2, \dotsc, n\}$, where 
	\[
	T^{e(v)}=T_1^{e_1(v)}\dotsc T_n^{e_n(v)} \quad \text{for}  \quad e_j(v)=\left\{
	\begin{array}{ll}
		1, & j \in v \\
		0, & j \notin v.\\
	\end{array}
	\right.	
	\]
	We shall end this Section by showing that every $\A_r^n$-isometry satisfies the Brehmer's positivity condition. To prove this, we need a couple of preparatory results of which one is due to Athavale \cite{Athavale} that provides a necessary and sufficient condition such that a commuting $n$-tuple of subnormal contractions admits a commuting normal extension. 
	\begin{lem}[\cite{Athavale}, Proposition 0]\label{Athavale} A tuple $\underline{S}=(S_1, \dotsc, S_n)$ is a subnormal tuple if and only if 
			\[
			\underset{\substack{0 \leq p_i \leq k_i\\ 1 \leq i \leq n}}{\sum}(-1)^{p_1+\dotsc+p_n}\binom{k_1}{p_1}\dotsc \binom{k_n}{p_n}S_1^{*p_1}\dotsc S_n^{*p_n}S_1^{p_1}\dotsc S_n^{p_n} \geq 0
			\]
			for any non-negative integers $k_1, \dotsc, k_n$.
	\end{lem} 
\begin{lem}\label{lem418}
	Let $V$ be an $\A_r$-isometry on a Hilbert space $\mathcal{H}$. Then 
		\begin{equation*}
		\begin{split}
			\underset{ 0 \leq p \leq k}{\sum}(-1)^p\binom{k}{p}V^{*p}V^p 
			=
			\left\{
			\begin{array}{ll}
				I, & k=0 \\
				(1-r^2)^{(k-1)}(I-V^*V), & k > 0.\\
			\end{array}
			\right.	
		\end{split}
	\end{equation*} 
\end{lem}
\begin{proof}
	Let $U$ be an $\mathbb{A}_r$-unitary acting on a space $\mathcal{K} \supseteq \mathcal{H}$ such that $V=U|_\mathcal{H}$. It follows from Theorem \ref{thm302} that  $\mathcal{K}$ decomposes into an orthogonal sum of reducing subspaces $\mathcal{K}=\mathcal{K}_u\oplus \mathcal{K}_r$ with repsect to which $U$ can be written as
	\[
	U=\begin{bmatrix}
		U_1 & 0\\
		0 & rU_2\\
	\end{bmatrix}
	\]
such that $U_1$ and $U_2$ are unitaries on $\mathcal{K}_u$ and $\mathcal{K}_r$ respectively. Consequently,  
	\begin{equation*}
		\begin{split}
			\underset{ 0 \leq p \leq k}{\sum}(-1)^p\binom{k}{p}U^{*p}U^p
			=
			\begin{bmatrix}
				\underset{ 0 \leq p \leq k}{\sum}(-1)^p\binom{k}{p} & 0\\
				0 &  \underset{ 0 \leq p \leq k}{\sum}\binom{k}{p}(-r^2)^p\\
			\end{bmatrix}
			&=\left\{
			\begin{array}{ll}
				\begin{bmatrix}
					I & 0\\
					0 & I\\
				\end{bmatrix}, & k=0 \\\\
				\begin{bmatrix}
					0 & 0\\
					0 & (1-r^2)^kI\\
				\end{bmatrix}, & k > 0.\\
			\end{array}
			\right. 
		\end{split}
	\end{equation*}
	Note that $I-U^*U=\begin{bmatrix}
		0 & 0\\
		0 & (1-r^2)I
	\end{bmatrix}$ and so, 
 $\begin{bmatrix}
		0 & 0\\
		0 & I
	\end{bmatrix}=(1-r^2)^{-1}(I-U^*U)$. This shows that
	\[
	\underset{ 0 \leq p \leq k}{\sum}(-1)^p\binom{k}{p}U^{*p}U^p= \left\{
	\begin{array}{ll}
		I, & k=0 \\
		(1-r^2)^{(k-1)}(I-U^*U), & k > 0.\\
	\end{array}
	\right. 
	\]
	Since $V=U|_\mathcal{H}$, we have that
	 $U=\begin{bmatrix}
		V & *\\
		0 & *\\
	\end{bmatrix}$ with respect to the decomposition $\mathcal{K}=\mathcal{H}\oplus \mathcal{H}^\perp$. Thus, for any $p \geq 0$, we have that 
	\[
	U^p=\begin{bmatrix}
		V^p & *\\
		0 & *\\
	\end{bmatrix}, \quad  
	U^{*p}=\begin{bmatrix}
		V^{*p} & 0\\
		* & *\\
	\end{bmatrix} \quad \mbox{and} \quad U^{*p}U^p=\begin{bmatrix}
		V^{*p}V^p & *\\
		* & *\\
	\end{bmatrix}.
	\]
	Consequently, we have
	\begin{equation*}
		\begin{split}
			\underset{ 0 \leq p \leq k}{\sum}(-1)^p\binom{k}{p}V^{*p}V^p 
			&=
			P_{\mathcal{H}}\left[\ \underset{ 0 \leq p \leq k}{\sum}(-1)^p\binom{k}{p}U^{*p}U^p\right]\bigg|_\mathcal{H}\\
			&=
			\left\{
			\begin{array}{ll}
				I, & k=0 \\
				(1-r^2)^{(k-1)}P_\mathcal{H}(I-U^*U)|_\mathcal{H}, & k > 0\\
			\end{array}
			\right.		\\
			&=
			\left\{
			\begin{array}{ll}
				I, & k=0 \\
				(1-r^2)^{(k-1)}(I-V^*V), & k > 0\\
			\end{array}
			\right.	
		\end{split}
	\end{equation*} 
and the proof is complete. 

\end{proof}
\begin{thm}
	Every $\A_r^n$-isometry satisfies the Brehmer's positivity condition.
\end{thm}	
\begin{proof}
	Let $\underline{V}=(V_1, \dotsc, V_n)$ be an $\A_r^n$-isometry acting on a Hilbert space $\mathcal{H}$. It follows from Theorem \ref{thm416} that each $V_i$ is an $\A_r$-isometry. For $1 \leq m \leq n$ and $k=(k_1, \dotsc, k_m) \in \mathbb{N}^m$, let us define the following: 
	\[
	\Delta_m^k(V_{i_1}, \dotsc, V_{i_m}):=\underset{\substack{0 \leq p_i \leq k_i \\ 1 \leq i \leq m}}{\sum}(-1)^{p_1+\dotsc+p_m}\binom{k_1}{p_1}\dotsc \binom{k_m}{p_m}V_{i_1}^{*p_1}\dotsc V_{i_m}^{*p_m}V_{i_1}^{p_1}\dotsc V_{i_m}^{p_m}\,,
	\] 
where $\{V_{i_1}, \dotsc, V_{i_m}\} \subset \{V_1, \dotsc, V_n\}$. By Lemma \ref{lem418}, $\Delta_1^k(V_i)=(1-r^2)^{k-1}(I-V_i^*V_i)$ for $k \in \N$ and $1 \leq i \leq n$. Next, we compute the following.
	\begin{equation*}
		\begin{split}
			\Delta_2^k(V_{i_1}, V_{i_2})
			&=
			\overset{k_1}{\underset{p_1=0}{\sum}}\overset{k_2}{\underset{p_2=0}{\sum}}(-1)^{p_1+p_2}\binom{k_1}{p_1}\binom{k_2}{p_2}V_{i_2}^{*p_2}V_{i_1}^{*p_1}V_{i_1}^{p_1}V_{i_2}^{p_2}\\
			&=
			\overset{k_2}{\underset{p_2=0}{\sum}}(-1)^{p_2}\binom{k_2}{p_2}V_{i_2}^{*p_2}\bigg[\overset{k_1}{\underset{p_1=0}{\sum}}(-1)^{p_1}\binom{k_1}{p_1}V_{i_1}^{*p_1}V_{i_1}^{p_1}\bigg]V_{i_2}^{p_2}\\
			&=
			(1-r^2)^{k_1-1}	\overset{k_2}{\underset{p_2=0}{\sum}}(-1)^{p_2}\binom{k_2}{p_2}V_{i_2}^{*p_2}(I-V_{i_1}^*V_{i_1})V_{i_2}^{p_2}\\
			&=
			(1-r^2)^{k_1-1}\bigg[	\overset{k_2}{\underset{p_2=0}{\sum}}(-1)^{p_2}\binom{k_2}{p_2}V_{i_2}^{*p_2}V_{i_2}^{p_2}-	V_{i_1}^*	\bigg(\overset{k_2}{\underset{p_2=0}{\sum}}(-1)^{p_2}\binom{k_2}{p_2}V_{i_2}^{*p_2}V_{i_2}^{p_2}\bigg)V_{i_1}\bigg]\\	
			&=
			(1-r^2)^{k_1-1}(1-r^2)^{k_2-1}\bigg[(I-V_{i_2}^*V_{i_2})-	V_{i_1}^*(I-V_{i_2}^*V_{i_2})V_{i_1}\bigg]\\	
			&=
			(1-r^2)^{k_1+k_2-2}\bigg[I-V_{i_1}^*V_{i_1}-V_{i_2}^*V_{i_2}+V_{i_1}^*V_{i_2}^*V_{i_1}V_{i_2}	\bigg].\\	
		\end{split}
	\end{equation*}  
	Similarly, one can show that $(1-r^2)^{-(k_1+k_2+k_3-3)}\Delta_3^k(V_{i_1}, V_{i_2}, V_{i_3})$ equals
	\[
I-V_{i_1}^*V_{i_1}-V_{i_2}^*V_{i_2}-V_{i_3}^*V_{i_3}+V_{i_1}^*V_{i_2}^*V_{i_1}V_{i_2}+V_{i_1}^*V_{i_3}^*V_{i_1}V_{i_3}+V_{i_2}^*V_{i_3}^*V_{i_2}V_{i_3}-V_{i_1}^*V_{i_2}^*V_{i_3}^*V_{i_1}V_{i_2}V_{i_3}.
	\]
	\noindent So far, we have shown that
	\begin{enumerate}
		\item $\Delta_1^k(V_{i_1})=(1-r^2)^{k_1-1}S(\{i_1\})$;
		\item $\Delta_2^k(V_{i_1}, V_{i_2}, V_{i_3})=(1-r^2)^{k_1+k_2-2}S(\{i_1, i_2\})$;
		\item $\Delta_3^k(V_{i_1}, V_{i_2}, V_{i_3})=(1-r^2)^{k_1+k_2+k_3-3}S(\{i_1, i_2, i_3\})$,
	\end{enumerate}
	where $S(u)$ is defined as in (\ref{Szego}) for a given subset $u$ of $\{1, \dotsc, n\}$. We provide an inductive argument to show that for every $m$ with $1 \leq m \leq n$ and $k=(k_1, \dotsc, k_m) \in \mathbb{N}^m,$ we have
	\begin{equation}\label{B.P.}
		\Delta_m^k(V_{i_1}, \dotsc, V_{i_m})=(1-r^2)^{k_1+ \dotsc + k_m-m}S(\{i_1, \dotsc, i_m\}).
	\end{equation}
	We assume that (\ref{B.P.}) holds for $m-1$ and prove it for $m$. To begin with, we have
	\begin{small}
		\begin{equation}\label{Delta_m}
			\begin{split}
				& \Delta_m^k(V_{i_1}, \dotsc, V_{i_m})\\
				&= 	\underset{\substack{0 \leq p_i \leq k_i\\ 1 \leq i \leq m}}{\sum}(-1)^{p_1+\dotsc+p_m}\binom{k_1}{p_1}\dotsc \binom{k_m}{p_m}V_{i_1}^{*p_1}\dotsc V_{i_m}^{*p_m}V_{i_1}^{p_1}\dotsc V_{i_m}^{p_m}	
				\\
				&=	\underset{0 \leq p_m \leq k_m}{\sum}(-1)^{p_m}\binom{k_m}{p_m}V_{i_m}^{*p_m}\bigg[ \underset{\substack{0 \leq p_i \leq k_i\\ 1 \leq i \leq m-1}}{\sum}(-1)^{p_1+\dotsc+p_{m-1}}\binom{k_1}{p_1}\dotsc \binom{k_{m-1}}{p_{m-1}}V_{i_1}^{*p_1}\dotsc V_{i_{m-1}}^{*p_{m-1}}V_{i_1}^{p_1}\dotsc V_{i_{m-1}}^{p_{m-1}}\bigg]V_{i_m}^{p_m}	
				\\
				&=	\underset{0 \leq p_m \leq k_m}{\sum}(-1)^{p_m}\binom{k_m}{p_m}V_{i_m}^{*p_m} \Delta_{m-1}^k(V_{i_1}, \dots , V_{i_{m-1}})V_{i_m}^{p_m}	
				\\
				&=(1-r^2)^{k_1+\dotsc +k_{m-1}-m+1} \\
				& \ \	\overset{k_m}{\underset{p_m=0}{\sum}}\binom{k_m}{p_m}(-V_{i_m})^{*p_m}\bigg[ I-\overset{m-1}{\underset{\alpha=1}{\sum}}V_{i_\alpha}^*V_{i_\alpha}+\underset{1 \leq \alpha< \beta \leq m-1}{\sum}V_{i_\alpha}^*V_{i_\beta}^*V_{i_\alpha}V_{i_\beta}+ \dotsc + (-1)^{m-1}V_{i_1}^*\dotsc V_{i_{m-1}}^*V_{i_1}\dotsc V_{i_{m-1}}\bigg]V_{i_m}^{p_m}.
				\\
			\end{split}
		\end{equation}
	\end{small}
	Again by Lemma \ref{lem418}, we have 
	\[
	\underset{0 \leq p_m \leq k_m}{\sum}(-1)^{p_m}\binom{k_m}{p_m}V_{i_m}^{*p_m}V_{i_m}^{p_m}=(1-r^2)^{k_m-1}(I-V_{i_m}^*V_{i_m})
	\]
	since $k_m \neq 0$. We have that 
	\begin{enumerate}
		\item[(A$_1$)] $	\underset{0 \leq p_m \leq k_m}{\sum}(-1)^{p_m}\binom{k_m}{p_m}V_{i_m}^{*p_m}\bigg[\overset{m-1}{\underset{\alpha=1}{\sum}}V_{i_\alpha}^*V_{i_\alpha}\bigg]V_{i_m}^{p_m}$
		\begin{equation*}
			\begin{split}
				&=
				\overset{m-1}{\underset{\alpha=1}{\sum}}V_{i_\alpha}^*	\bigg[\underset{0 \leq p_m \leq k_m}{\sum}(-1)^{p_m}\binom{k_m}{p_m}V_{i_m}^{*p_m}V_{i_m}^{p_m}\bigg]V_{i_\alpha}\\
				&=
				(1-r^2)^{k_m-1}	\overset{m-1}{\underset{\alpha=1}{\sum}}V_{i_\alpha}^*	\bigg(I-V_{i_m}^*V_{i_m}\bigg)V_{i_\alpha}\\
				&=
				(1-r^2)^{k_m-1}\bigg[ 	\overset{m-1}{\underset{\alpha=1}{\sum}}V_{i_\alpha}^*V_{i_\alpha}	-\overset{m-1}{\underset{\alpha=1}{\sum}}V_{i_\alpha}^*V_{i_m}^*V_{i_\alpha}V_{i_m}\bigg].\\
			\end{split}
		\end{equation*}
		\item[(A$_2$)] $ \underset{0 \leq p_m \leq k_m}{\sum}(-1)^{p_m}\binom{k_m}{p_m}V_{i_m}^{*p_m}\bigg[\underset{1 \leq \alpha< \beta \leq m-1}{\sum}V_{i_\alpha}^*V_{i_\beta}^*V_{i_\alpha}V_{i_\beta}\bigg]V_{i_m}^{p_m}$ 
		\begin{equation*}
			\begin{split}
				&=
				\underset{1 \leq \alpha< \beta \leq m-1}{\sum}V_{i_\alpha}^*V_{i_\beta}^*	 \bigg[\underset{0 \leq p_m \leq k_m}{\sum}(-1)^{p_m}\binom{k_m}{p_m}V_{i_m}^{*p_m}V_{i_m}^{p_m}\bigg]V_{i_\alpha}V_{i_\beta}\\
				&=
				(1-r^2)^{k_m-1}
				\underset{1 \leq \alpha< \beta \leq m-1}{\sum}V_{i_\alpha}^*V_{i_\beta}^*	 \bigg(I-V_{i_m}^*V_{i_m}\bigg)V_{i_\alpha}V_{i_\beta}\\
				&=
				(1-r^2)^{k_m-1}\bigg[ 
				\underset{1 \leq \alpha< \beta \leq m-1}{\sum}V_{i_\alpha}^*V_{i_\beta}^*V_{i_\alpha}V_{i_\beta}
				-
				\underset{1 \leq \alpha< \beta \leq m-1}{\sum}V_{i_\alpha}^*V_{i_\beta}^*	 V_{i_m}^*V_{i_m}V_{i_\alpha}V_{i_\beta}
				\bigg]. \\
			\end{split}
		\end{equation*}
		\vdots 
		\\
		\vdots
		\item[(A$_{m-1}$)] $ \underset{0 \leq p_m \leq k_m}{\sum}(-1)^{p_m}\binom{k_m}{p_m}V_{i_m}^{*p_m}\bigg[V_{i_1}^*\dotsc V_{i_{m-1}}^*V_{i_1}\dotsc V_{i_{m-1}}\bigg]V_{i_m}^{p_m}$
		\begin{equation*}
			\begin{split}
				&=
				V_{i_1}^*\dotsc V_{i_{m-1}}^*\bigg[\underset{0 \leq p_m \leq k_m}{\sum}(-1)^{p_m}\binom{k_m}{p_m}V_{i_m}^{*p_m}V_{i_m}^{p_m}\bigg]V_{i_1}\dotsc V_{i_{m-1}}\\
				&=
				(1-r^2)^{k_m-1}
				V_{i_1}^*\dotsc V_{i_{m-1}}^*\bigg(I-V_{i_m}^*V_{i_m}\bigg)V_{i_1}\dotsc V_{i_{m-1}}\\
				&=		
				(1-r^2)^{k_m-1}
				\bigg[		V_{i_1}^*\dotsc V_{i_{m-1}}^*V_{i_1}\dotsc V_{i_{m-1}}
				-V_{i_1}^*\dotsc V_{i_{m-1}}^*V_{i_m}^*V_{i_1}\dotsc V_{i_{m-1}}V_{i_m}
				\bigg].		
				\\
			\end{split}
		\end{equation*}
	\end{enumerate}
	Putting the above computations in (\ref{Delta_m}), we have that 
	\[
	(1-r^2)^{-(k_1+\dotsc +k_m-m)}\Delta_m^k(V_{i_1}, \dotsc, V_{i_m})
	\]
	\begin{small} 
		\begin{equation*}
			\begin{split}
				&=\overset{k_m}{\underset{p_m=0}{\sum}}\binom{k_m}{p_m}(-V_{i_m}^*)^{p_m}\bigg[ I-\overset{m-1}{\underset{\alpha=1}{\sum}}V_{i_\alpha}^*V_{i_\alpha}+\underset{1 \leq \alpha< \beta \leq m-1}{\sum}V_{i_\alpha}^*V_{i_\beta}^*V_{i_\alpha}V_{i_\beta}+ \dotsc + (-1)^{m-1}V_{i_1}^*\dotsc V_{i_{m-1}}^*V_{i_1}\dotsc V_{i_{m-1}}\bigg]V_{i_m}^{p_m}
				\\
				&= I-\overset{m}{\underset{\alpha=1}{\sum}}V_{i_\alpha}^*V_{i_\alpha}+\underset{1 \leq \alpha< \beta \leq m}{\sum}V_{i_\alpha}^*V_{i_\beta}^*V_{i_\alpha}V_{i_\beta}+ \dotsc + (-1)^{m}V_{i_1}^*\dotsc V_{i_{m}}^*V_{i_1}\dotsc V_{i_{m}}\\
				&=S(\{i_1, \dotsc, i_m\}).	
			\end{split}
		\end{equation*}
	\end{small} 

\vspace{0.1cm}

\noindent 	Hence, we have proved that (\ref{B.P.}) holds for every $m \in \mathbb{N}$. It follows from Lemma \ref{Athavale} that $\Delta_m^k(V_{i_1}, \dotsc, V_{i_m}) \geq 0$ for every $k \in\N$ and $1 \leq m \leq n$. Thus, by (\ref{B.P.}) we have that $S(\{i_1, \dotsc, i_m\})\geq 0$ for every $\{i_1, \dotsc, i_m\} \subseteq \{1, \dotsc, n\}$. The proof is now complete.  
\end{proof}


		\section{Canonical decomposition of an $\mathbb{A}_r$-contraction}\label{annulus contraction}
		
		\vspace{0.2cm}
		
\noindent The next step to the Wold decomposition of an isometry is the canonical decomposition of a contraction as in Theorem \ref{thm104}. Indeed, for any contraction $T$ on a Hilbert space $\mathcal{H}$, the space $\mathcal{H}$ admits an orthogonal decomposition $\mathcal{H}_u\oplus \mathcal{H}_c$ into reducing subspaces of $T$ such that $T|_{\mathcal{H}_u}$ is a unitary and  $T|_{\mathcal{H}_c}$ is a completely non-unitary (c.n.u.) contraction. Note that $\mathcal{H}_u$ is the maximal reducing subspace on which $T$ acts as a unitary and it is given by 
\[
\mathcal{H}_u =\{h \in \mathcal{H}: \|T^nh\|=\|h\|=\|T^{*n}h\|, \ \ n= 0, 1, 2, \dotsc\}.
\]
Recently, Majee and Maji (see  Theorem 2.2 in \cite{Maji}) expressed the subspaces $\HS_u$ and $\HS_c$ in the following way:
\begin{equation}\label{Maji}
\HS_u=\overset{\infty}{\underset{k=0}{\bigcap}}\left[Ker(I-T^{*k}T^k)\cap Ker(I-T^kT^{*k})\right], \quad 	
	\mathcal{H}_c 
	=\overset{\infty}{\underset{k=0}{\bigvee}}\left\{Ran(I-T^{*k}T^k)\cup Ran(I-T^kT^{*k})\right\}.
\end{equation} 

\noindent Astonishingly, if $\widetilde{T}$ is another contraction on $\mathcal{H}$ that doubly commutes with $T$, then $\mathcal{H}_u , \HS_c$ reduce $\widetilde{T}$ as the following theorem reveals.

\begin{thm}[\cite{Slocinski-I} \& \cite{Pal}, Theorem 4.2]  \label{dc}
	For a doubly commuting pair of operators $T$ and $\widetilde{T}$ acting on a Hilbert space $\mathcal{H}$ such that $T$ is a contraction, let $\mathcal{H}=\mathcal{H}_u\oplus \mathcal{H}_c$ be the canonical decomposition of $T$. Then $\mathcal{H}_u, \mathcal{H}_c$ are reducing subspaces for $\widetilde{T}$.
\end{thm}

\noindent In this Section, we find an analogous canonical decomposition for an $\A_r$-contraction and more interestingly such a decomposition follows from the canonical decomposition of a contraction. First, we define the natural classes of $\A_r$-contractions.
\begin{defn}
	An $\A_r$-contraction $T$ acting on a Hilbert space $\mathcal{H}$ is said to be 
	\begin{enumerate}
		\item  a \textit{completely non-unitary} $\mathbb{A}_r$-\textit{contraction} or simply a \textit{c.n.u.} $\mathbb{A}_r$-\textit{contraction} if there is no non-zero closed linear subspace of $\mathcal{H}$ that reduces $T$ and on which $T$ acts as an $\mathbb{A}_r$-unitary;
		\item a \textit{completely non-isometry} $\mathbb{A}_r$-\textit{contraction} or simply a \textit{c.n.i. $\mathbb{A}_r$-contraction} if there is no non-zero closed linear subspace of $\mathcal{H}$ that reduces $T$ and on which $T$ acts as an $\mathbb{A}_r$-isometry.
	\end{enumerate}
\end{defn}
Below we mention a basic and simple result on an $\A_r$-contraction which will be used frequently in sequel.
 
\begin{lem}[\cite{N-S1}, Lemma 2.5]\label{lem503}
	Let $T$ be an $\A_r$-contraction. Then $rT^{-1}$ is also an $\A_r$-contraction. Moreover, both $T$ and $rT^{-1}$ are contractions. 
\end{lem}
\begin{rem}\label{implications} The following diagram sums up the relations between the operators associated with $\DC$ and $\CA_r$ :
\[
\xymatrix{
	\mathbb{A}_r\mbox{-unitary} \ \ \ar@{=>}[r] & \ \ 
	\mathbb{A}_r\mbox{-isometry} \ar@{=>}[r] & \ \
	\mathbb{A}_r\mbox{-contraction} \ar@{=>}[d]&\\
	\mbox{unitary}  \ \ \ar@{=>}[u] \ \ \ar@{=>}[r]& \ \ 
	\mbox{isometry}  \ar@{=>}[r] & \ \
	\mbox{contraction} .\\
}
\]
\end{rem}		
\begin{thm}
	\label{canonical decomposition}$($\textit{Canonical decomposition of an $\mathbb{A}_r$-contraction}$)$ For every $\mathbb{A}_r$-contraction $T$ on a Hilbert space $\mathcal{H},$ the space $\HS$ admits a unique orthogonal decomposition into closed reducing subspaces of $T$, say $\mathcal{H}= \mathcal{H}_{a} \oplus \mathcal{H}_{c},$ such that $T|_{\mathcal{H}_a}$ is an $\mathbb{A}_r$-unitary and $T|_{\mathcal{H}_c}$ is a c.n.u. $\mathbb{A}_r$-contraction; $\mathcal{H}_{a}$ or $\mathcal{H}_{c}$ may equal the trivial subspace $\{0\}$. Moreover, the spaces can be formulated as $\mathcal{H}_a=\mathcal{H}_u \oplus \mathcal{H}_r$ and $\mathcal{H}_c=\mathcal{H}_{c1} \cap \mathcal{H}_{c2},$ where,
	\begin{align*}
		\mathcal{H}_u
		&=\overset{\infty}{\underset{k=0}{\bigcap}}\left[Ker(I-T^{*k}T^k)\cap Ker(I-T^kT^{*k})\right],\\
		\mathcal{H}_r 
		&=\overset{\infty}{\underset{k=0}{\bigcap}}\left[Ker\left(I-(rT^{-1})^{*k}(rT^{-1})^k\right)\cap Ker\left(I-(rT^{-1})^k(rT^{-1})^{*k}\right)\right],\\
		\mathcal{H}_{c1}  
		&= \overset{\infty}{\underset{k=0}{\bigvee}}\left\{Ran(I-T^{*k}T^k)\cup Ran(I-T^kT^{*k})\right\}, \\
		\mathcal{H}_{c2} &=\overset{\infty}{\underset{k=0}{\bigvee}}\left\{Ran(I-(rT^{-1})^{*k}(rT^{-1})^k) \cup Ran(I-(rT^{-1})^k(rT^{-1})^{*k})\right\}.
	\end{align*}
\end{thm}
\begin{proof}
	 The proof runs along the same lines as the proof of Theorem \ref{Wold decomposition}. We mention a few important steps here for the sake of completeness.  The spaces
	\begin{align*}
	\mathcal{H}_u
	&=\left\{h \in \mathcal{H} \ : \|T^nh\|=\|h\|=\|T^{*n}h\|, \ \ n= 0, 1, 2,\dotsc\right\} \quad  \text{and}\\
	\mathcal{H}_r 
	&=\left\{h \in \mathcal{H} \ : \ \|(rT^{-1})^nh\|=\|h\|=\|(rT^{-1})^{*n}h\|, \  \ n = 0, 1, 2,\dotsc \right\}.
\end{align*}
are closed subspaces of $\mathcal{H}$ reducing $T, T^{-1}$ which follows from Lemma \ref{lem412}. Moreover, $T|_{\mathcal{H}_u}$ and $rT^{-1}|_{\mathcal{H}_r}$ are unitaries. Let $\mathcal{H}_a=\mathcal{H}_u \oplus \mathcal{H}_r$ and $\mathcal{H}_c=\mathcal{H}\ominus\mathcal{H}_a$. Then we have an orthogonal decomposition $\mathcal{H}=\mathcal{H}_{a}\oplus \mathcal{H}_{c}$ such that $\mathcal{H}_a, \mathcal{H}_c$ reduce $T$. It follows from Theorem \ref{thm302} that $T|_{\mathcal{H}_a}$ is an $\A_r$-unitary. We claim that  $T|_{\mathcal{H}_c}$ is a c.n.u. $\A_r$-unitary. Let $\mathcal{H}_{ca}$ be a closed subspace of $\mathcal{H}$ reducing $T$ such that $T|_{\mathcal{H}_{ca}}$ is an $\mathbb{A}_r$-unitary.
	By Theorem \ref{thm302}, we have an orthogonal decomposition of $ \mathcal{H}_{ca},$ say $\mathcal{H}_{ca}=\mathcal{H}_{cu}\oplus \mathcal{H}_{cr}$, such that $T|_{\mathcal{H}_{cu}}$ and $rT^{-1}|_{\mathcal{H}_{cr}}$ are unitaries. From the definition of $\mathcal{H}_u$ and $\mathcal{H}_r$, it follows that $\mathcal{H}_{cu} \subseteq \mathcal{H}_u$ and $\mathcal{H}_{cr} \subseteq \mathcal{H}_r$. Thus,
$
 \mathcal{H}_{ca} \subseteq \mathcal{H}_u \oplus \mathcal{H}_r =\mathcal{H}_{a}
$
which shows that $\mathcal{H}_a$ is the maximal closed subspace of $\mathcal{H}$ that reduces $T$ to an $\A_r$-uniatry. Consequently, $ T|_{\mathcal{H}_c}$ is a c.n.u. $\mathbb{A}_r$-contraction. Assume that $\mathcal{H}=\mathcal{L}_a \oplus \mathcal{L}_c$ be such that $\mathcal{L}_a$ and $ \mathcal{L}_c$ reduce $T$ to an $\mathbb{A}_r$-unitary and a c.n.u. $\mathbb{A}_r$-contraction respectively. Since $\mathcal{H}_a$ is the maximal closed  subspace of $\mathcal{H}$ that reduces $T$ to an $\A_r$-unitary, we have that $\mathcal{L}_{a} \subseteq \mathcal{H}_a$. Then $T|_{\mathcal{H}_a\ominus \mathcal{L}_{a}}$ is an $\mathbb{A}_r$-unitary. Note that $\mathcal{H}_a \ominus \mathcal{L}_a \subseteq \mathcal{H} \ominus \mathcal{L}_a =\mathcal{L}_c$. Since $T|_{\mathcal{L}_c}$ is a c.n.u. $\mathbb{A}_r$-contraction, we must have $\mathcal{H}_a \ominus \mathcal{L}_a=\{0\}$. Consequently, $\mathcal{H}_a= \mathcal{L}_a$ and $\mathcal{H}_c= \mathcal{L}_c$. Thus, the decomposition is unique.
	
	\vspace{0.2cm}
	
\noindent Now we determine the desired formulation of the spaces. It follows from (\ref{Maji}) that
\[
	\mathcal{H}_u
=\overset{\infty}{\underset{k=0}{\bigcap}}\left[Ker(I-T^{*k}T^k)\cap Ker(I-T^kT^{*k})\right]
.
\]
Similarly, applying (\ref{Maji}) to the contraction $rT^{-1}$, we get that
\[
\mathcal{H}_r 
=\overset{\infty}{\underset{k=0}{\bigcap}}\left[Ker\left(I-(rT^{-1})^{*k}(rT^{-1})^k\right)\cap Ker\left(I-(rT^{-1})^k(rT^{-1})^{*k}\right)\right].
\] 
Since $\mathcal{H}_a=\mathcal{H}_u \oplus \mathcal{H}_r$ and $\mathcal{H}_c=\mathcal{H}\ominus(\mathcal{H}_u \oplus \mathcal{H}_r)$, we have that $\mathcal{H}_c=\mathcal{H}_u^{\perp} \cap \mathcal{H}_r^{\perp}$. Some standard calculation and (\ref{Maji}) yields that
	\begin{align*}
			\mathcal{H}_u^{\perp}
			&=\left[ \overset{\infty}{\underset{k=0}{\bigcap}}[Ker(I-T^{*k}T^k)\bigcap Ker(I-T^kT^{*k})] \right]^{\perp}
			=\overset{\infty}{\underset{k=0}{\bigvee}}\left\{Ran(I-T^{*k}T^k)\bigcup Ran(I-T^kT^{*k})\right\}
			\end{align*}	
		and
			\begin{align*}
			\mathcal{H}_r^{\perp}
			&=\left[ \overset{\infty}{\underset{k=0}{\bigcap}}[Ker\left(I-(rT^{-1})^{*k}(rT^{-1})^k\right)\bigcap Ker\left(I-(rT^{-1})^k(rT^{-1})^{*k}\right)]\right]^{\perp}\\
			&=\overset{\infty}{\underset{k=0}{\bigvee}}\left\{Ran(I-(rT^{-1})^{*k}(rT^{-1})^k) \cup Ran(I-(rT^{-1})^{k}(rT^{-1})^{*k})\right\}.
	\end{align*}
	This completes the proof.
\end{proof}

Now it is clear that for an $\mathbb{A}_r$-contraction $T$, the contractions $T$ and $rT^{-1}$ play pivotal role in deciphering the structure of $T$. The rest of this Section will be devoted to study the class of c.n.u. $\mathbb{A}_r$-contractions. Indeed, we obtain necessary and sufficient conditions such that an $\A_r$-contraction $T$ becomes a c.n.u. $\A_r$-contraction with the help of the commuting pair $(T, rT^{-1})$.
\begin{thm}\label{cnu contractions}
	Let $T \in \mathcal{B}(\mathcal{H})$ be an $\mathbb{A}_r$-contraction. Then the following are equivalent : 
	\begin{enumerate}
		\item $T$ is a c.n.u. $\mathbb{A}_r$-contraction; 
		\item $T$ and $rT^{-1}$ are c.n.u. contractions;
		\item $T$ and $rT^{-1}$ are c.n.i. contractions.
	\end{enumerate}
\end{thm}
\begin{proof} We prove $(1) \iff (2) \iff (3)$. Since $T$ is an $\A_r$-contraction, it follows from Lemma \ref{lem503} that $rT^{-1}$ is an $\A_r$-contraction. Moreover, $T$ and $rT^{-1}$ both are contractions. 

	\vspace{0.2cm}

\noindent $(1) \implies (2).$
		 Assume that $T$ is a c.n.u. $\mathbb{A}_r$-contraction. Let if possible there be a closed subspace $\mathcal{L}$ of $\mathcal{H}$ reducing $T$ such that either $T|_{\mathcal{L}}$ or $rT^{-1}|_{\mathcal{L}}$ is a unitary. In either case, it follows from Theorem \ref{thm302} that $T|_\mathcal{L}$ is an $\mathbb{A}_r$-unitary. Consequently, $\mathcal{L}=\{0\}$ and so, both $T$ and $rT^{-1}$ are c.n.u. contractions.

\vspace{0.2cm}

\noindent $(2) \implies (1).$		 
	Let $T$ and $rT^{-1}$ be c.n.u. contractions. Let if possible, there exists a closed subspace $\mathcal{H}_a$ of $\mathcal{H}$ that reduces $T$ to an $\mathbb{A}_r$-unitary. By Theorem \ref{thm302}, we have the orthogonal decomposition of $\mathcal{H}_a$ into closed subspaces $\mathcal{H}_u$ and $\mathcal{H}_r$ reducing $T$ such that $T|_{\mathcal{H}_u}$ and $rT^{-1}|_{\mathcal{H}_r}$ are unitaries. Consequently, $\mathcal{H}_u=\mathcal{H}_r=\{0\}$ and so, $\mathcal{H}_a=\{0\}$. Thus, $T$ is a c.n.u. $\A_r$-contraction.
	
\vspace{0.2cm}
	
\noindent $(2) \implies (3).$ Suppose $T$ and $rT^{-1}$ are c.n.u. contractions.
By Theorem \ref{Levan}, $\mathcal{H}$ decomposes into an orthogonal sum $\mathcal{H}=\mathcal{H}_1 \oplus \mathcal{H}_2$ such that $\mathcal{H}_1$ and $\mathcal{H}_2$ reduce $T$ to an isometry and a c.n.i. contraction respectively. It follows from the invertibility of $T$ and Lemma \ref{lem412} that $T|_{\mathcal{H}_1}$ is invertible and so is a unitary. Since $T$ is a c.n.u. contraction, we have $\mathcal{H}_1=\{0\}$. Thus, $\mathcal{H}=\mathcal{H}_2$ and hence, $T$ is a c.n.i. contraction. Similarly, one can prove that $rT^{-1}$ is a c.n.i. contraction.

\vspace{0.2cm}

\noindent $(3) \implies (2).$ Let $T$ and $rT^{-1}$ be c.n.i. contractions. Then $T$ and $rT^{-1}$ are evidently c.n.u. contractions. This finishes the proof.

\end{proof}
A celebrated theorem due to N. Levan (see Theorem \ref{Levan}) states that every c.n.u. contraction can be orthogonally decomposed into a pure isometry (i.e. a shift operator) and a c.n.i. contraction. We now find an analogue of this result for an $\A_r$-contraction. The idea of our proof is similar to that of Theorem 1 in \cite{Levan}.
\begin{thm}
	\label{Levan type decomposition} $($\textit{Decomposition of a c.n.u. $\mathbb{A}_r$-contraction}$)$ For every c.n.u. $\mathbb{A}_r$-contraction $T$ on a Hilbert space $\mathcal{H}$, the space $\mathcal{H}$ admits a unique orthognal decomposition $\mathcal{H}= \mathcal{H}_{1} \oplus \mathcal{H}_{2}$ into reducing subspaces of $T$ such that the $T|_{\mathcal{H}_{1}}$ is a pure $\mathbb{A}_r$-isometry and $T|_{\mathcal{H}_{2}}$ is a c.n.i. $\mathbb{A}_r$-contraction. 
\end{thm}
\begin{proof}
 Let $\mathcal{H}_1'$ be the span of all closed subspaces of $\mathcal{H}$ reducing $T$ restricted to which $T$ acts as an $\mathbb{A}_r$-isometry and let $\mathcal{H}_1$ be the closure of $\mathcal{H}_1'$ in $\HS$. Then $\mathcal{H}_1$ reduces $T$ and so, $\mathcal{H}$ admits the orthogonal decomposition $\mathcal{H}=\mathcal{H}_1 \oplus \mathcal{H}_2$. We show that $V=T|_{\mathcal{H}_1}$ is a pure $\A_r$-isometry. For the ease of computations, we let $\Lambda$ be the collection of all closed subspaces of $\mathcal{H}$ reducing $T$ to an $\A_r$-isometry. Take any $x \in \mathcal{H}_1'$. Then there exists $x_j \in \mathcal{L}_j \in \Lambda \ (1 \leq j \leq m)$ such that $x=x_1 + \dotsc + x_m$. By Theorem \ref{thm407}, we have
	$
	-V^{*2}V^2x_j+(1+r^2)V^*Vx_j-r^2x_j=0
	$
for every $1 \leq j \leq m$ and thus, $-V^{*2}V^2x+(1+r^2)V^*Vx-r^2x=0$. By continuity arguments, one can easily show that $-V^{*2}V^2x+(1+r^2)V^*Vx-r^2x=0$ for every $x \in \overline{\mathcal{H}_1'}=\mathcal{H}_1$. Again, it follows from Theorem \ref{thm407} that $T|_{\mathcal{H}_1}$ is an $\A_r$-isometry. Also, since $T$ is a c.n.u. contraction, it does not have any $\A_r$-unitary part and thus, $T|_{\mathcal{H}_1}$ is a pure $\A_r$-isometry. Consequently, $\mathcal{H}_1$ is the largest closed reducing subspace that reduces $T$ to a pure $\A_r$-isometry and so, $T|_{\mathcal{H}_2}$ is a c.n.i. $\mathbb{A}_r$-contraction.
	
\smallskip	
	
 To see the uniqueness, let us assume that $\mathcal{H}=\mathcal{L}_{1} \oplus \mathcal{L}_{2}$ such that $\mathcal{L}_1$ and $\mathcal{L}_2$ reduce $T$ to a pure $\mathbb{A}_r$-isometry and a c.n.i. $\mathbb{A}_r$-contraction respectively. By the definition of $\mathcal{H}_1$, it follows that $\mathcal{L}_1 \subseteq \mathcal{H}_1$ and so,  $\mathcal{H}_1 \ominus \mathcal{L}_1$ reduces $T$ to a pure $\A_r$-isometry. Since $\mathcal{H}_1 \ominus \mathcal{L}_1 \subseteq \mathcal{H} \ominus \mathcal{L}_1=\mathcal{L}_2$ and $T|_{\mathcal{L}_2}$ is a c.n.i. $\mathbb{A}_r$-contraction, we have  $\mathcal{H}_1 \ominus \mathcal{L}_1 =\{0\}$. Thus $\mathcal{H}_1 = \mathcal{L}_1, \mathcal{H}_2= \mathcal{L}_2$ and the proof is complete.
\end{proof}

\vspace{0.1cm}
	\section{Canonical decomposition of finitely many doubly commuting $\A_r$-contractions}\label{doubly commuting}
	
	\vspace{0.2cm}
	\noindent It follows from the definitions that if $(T_1, \dots , T_n)$ is an $\A_r^n$-contraction, then each of $T_1, \dots , T_n$ is an $\A_r$-contraction. Therefore, the class of commuting $n$ number of $\A_r$-contractions is larger than or equal to the class of $\A_r^n$-contractions. It is unlikely to have a converse to the above statement, which is to say that for a tuple of commuting $\A_r$-contractions $(T_1, \dots , T_n)$, we should not expect that $(T_1, \dots , T_n)$ will always become an $\A_r^n$-contraction. At this point we do not have a counter example to this, but the reason behind such a hunch is that it does not hold for the polydisc $\D^n$. Indeed, the famous example due to  Kaijser and Varopoulos (see \cite{Paulsen}, Example 5.7) shows that a commuting triple of contractions $(A_1,A_2,A_3)$ may fail to have $\overline{\D}^3$ as a spectral set because of the failure of von Neumann's inequality, whereas each of $A_1,A_2,A_3$ has $\overline{\D}$ as a spectral set. For this reason, we shall always consider $n$ commuting (or, doubly commuting) $\A_r$-contractions so to deal with a wider possible class than the $\A_r^n$-contractions while finding joint invariant or reducing subspaces and orthogonal decompositions.

\smallskip
	
A canonical decomposition does not hold in general for commuting contractions, i.e. for a pair of commuting contractions $A_1,A_2$ acting on a Hilbert space $\HS$, not always one can find common reducing subspaces $\HS_1,\HS_2 \,\subset \HS$ of $A_1,A_2$ such that $A_1|_{\HS_1}, A_2|_{\HS_1}$ are unitaries and $A_1|_{\HS_2}, A_2|_{\HS_2}$ are both c.n.u. contractions. In fact, commuting contractions are far reaching, we do not even have such a Wold decomposition for a pair of commuting isometries, e.g. see \cite{Slocinski}. Berger, Coburn and Lebow \cite{BCL} exhibited a Wold-type decomposition to finitely many commuting isometries $V_1, \dots , V_n$ with respect to the Wold decomposition of their product $V=\prod_{i=1}^n \,V_i$. However, in general there is no nonzero common reducing subspace $\HS_c \subset \HS$ of $V_1, \dots , V_n$ on which all of them act as pure isometries as shown in \cite{Slocinski}. Hence, S\l oci\'{n}ski, Burdak, Kosiek and a few others made efforts to find appropriate generalization of Wold and canonical decompositions for commuting and doubly commuting contractions, e.g. see \cite{Slocinski, BurdakI, BurdakIII, BurdakII} and the references therein. Recently, the first named author of this article generalize these results for any family of doubly commuting contractions in \cite{Pal}. The motivation behind exploring Wold or canonical decompositions for doubly commuting contractions naturally leads to the possibility of these decompositions for doubly commuting $\A_r$-isometries and $\A_r$-contractions.

\smallskip
	 
	 In this Section, we generalize the canonical decomposition as outlined in Theorem \ref{canonical decomposition} to any finite set of doubly commuting $\A_r$-contractions. The following lemma plays the central role in the proof of this result and also in the rest of the paper.
	
		\begin{lem} \label{reducing}
			For doubly commuting $\A_r$-contractions $T_1, T_2$  acting on a Hilbert space $\mathcal{H}$, if $\mathcal{H}=\mathcal{H}_a \oplus \mathcal{H}_c$ is the canonical decomposition of $T_1$ as in Theorem \ref{canonical decomposition}, then $\mathcal{H}_a, \mathcal{H}_c$ are reducing subspaces for $T_2$.
	\end{lem}
	\begin{proof} 
		It suffices to show that $\mathcal{H}_a$ reduces $T_2$. It follows from Theorem \ref{canonical decomposition} that $\mathcal{H}_a=\mathcal{H}_u \oplus \mathcal{H}_r$ such that
		\begin{align*}
		\mathcal{H}_u
		&=\left\{h \in \mathcal{H} \ : \|T^nh\|=\|h\|=\|T^{*n}h\|, \ \ n= 0, 1, 2,\dotsc\right\} \quad  \text{and}\\
		\mathcal{H}_r 
		&=\left\{h \in \mathcal{H} \ : \ \|(rT^{-1})^nh\|=\|h\|=\|(rT^{-1})^{*n}h\|, \  \ n = 0, 1, 2,\dotsc \right\}.
	\end{align*}
	Note that the spaces $\mathcal{H}_u$ and $\mathcal{H}_r$ can be written as
	\[
	\mathcal{H}_u= \underset{n \in \mathbb{Z}\setminus\{0\}}{\bigcap} \mbox{Ker} \ D_{T_1^n} \quad \text{and} \quad \mathcal{H}_r= \underset{n \in \mathbb{Z}\setminus\{0\}}{\bigcap} \mbox{Ker} \ D_{(rT_1^{-1})^n}
	\]
	where 
	\[
D_{T^n}:= \left\{
	\begin{array}{ll}
		(I-T^{*n}T^n)^{1\slash 2}, & n \geq 0 \\ \\
		(I-T^{|n|}T^{*|n|})^{1\slash 2}, & n \leq 0\\
	\end{array} 
	\right.
	\]
	for a contraction $T$. Since $T_1$ and $T_2$ doubly commute with each other, $T_2$ and $T_2^*$ commute with $(I-T_1^{*n}T_1^n)$ and $(I-T_1^{|n|}T_1^{*|n|})$ for every $n \geq 0.$ Consequently, $T_2$ and $T_2^*$ commute with $D_{T_1^n}$ for every $n \ne 0.$ If $x \in$ Ker $D_{T_1^n}$, then $D_{T_1^n}T_2x=T_2D_{T_1^n}x=0$ and so, $T_2x \in$ Ker $D_{T_1^n}.$ Similarly, $T_2^*x \in$ Ker $D_{T_1^n}$ for every $x \in$ Ker $D_{T_1^n}.$ We have that for every $n \ne 0,$ Ker $D_{T_1^n}$ is a reducing subspace for $T_2$. Therefore, $\mathcal{H}_u$ reduces $T_2.$ Similarly, we can prove that $\mathcal{H}_r$ reduces $T_2$.  Consequently, $\mathcal{H}_a$ reduces $T_2$ and we get the desired conclusion.
	\end{proof}
	
	We shall use the following terminologies throughout the rest of the paper.
\begin{defn} \label{defn:061}
	Let $T$ be an $\A_r$-contraction acting on a Hilbert space $\HS$.
	\begin{enumerate}
		\item[(i)] $T$ is an \textit{atom} if $T$ is either an $\A_r$-unitary or a c.n.u. $\A_r$-contraction.
		\item[(ii)] $T$ is an \textit{atom of type $t_u$} if $T$ is an $\A_r$-unitary.
		\item[(iii)] $T$ is an \textit{atom of type $t_c$} if $T$ is a c.n.u. $\A_r$-contraction. 
		\item[(iv)] $T$ is a \textit{non-atom} if $T$ is neither an $\A_r$-unitary nor a c.n.u. $\A_r$-contraction.
	\end{enumerate}	
\end{defn}

Now we present a main result of this Section, a canonical decomposition for a finite family of doubly commuting contractions.
	\begin{thm}\label{thm602}
		Let $(T_1, \dotsc, T_n)$ be a doubly commuting tuple of $\mathbb{A}_r$-contractions acting on a Hilbert space $\mathcal{H}$. Then there exists a unique decomposition of $\HS$ into an orthogonal sum of $2^n$ subspaces $\HS_1, \dotsc, \HS_{2^n}$ of $\HS$ such that
		\begin{enumerate}
			\item every $\HS_j
			, 1 \leq j \leq 2^n$, is a joint reducing subspace for $T_1, \dotsc, T_n$ ;
			\item $T_i|_{\HS_j}$ is either an $\A_r$-unitary or a c.n.u. $\A_r$-contraction for $ 1 \leq i \leq n$ and $1 \leq j \leq 2^n$;
			\item if $\Phi_n$ is the collection of all functions $\phi:\{1, \dotsc, n\} \to \{t_u, t_c\}$, then for every $\phi \in \Phi_n$, there is a unique maximal joint reducing subspace $\HS_\ell (1 \leq \ell \leq 2^n)$ such that the $j$-th component of $(T_1|_{\HS_{\ell}}, \dotsc, T_n|_{\HS_{\ell}})$ is of the type $\phi(j)$ for $1 \leq j \leq n$;
			\item $\HS_1$ is the maximal joint reducing subspace for $T_1, \dotsc, T_n$ such that $T_1|_{\HS_1}, \dotsc, T_n|_{\HS_1}$ are $\A_r$-unitaries;
			\item $\HS_{2^n}$ is the maximal joint reducing subspace for $T_1, \dotsc, T_n$ such that $T_1|_{\HS_{2^n}}, \dotsc, T_n|_{\HS_{2^n}}$ are c.n.u. $\A_r$-contractions.
		\end{enumerate}
 One or more subspaces in the decomposition  may coincide with the subspace $\{0\}$.
	\end{thm}
		\begin{proof}
		We prove this by mathematical induction on $n$, the number of doubly commuting $\A_r$-contractions. The $n=1$ case follows from Theorem \ref{canonical decomposition}. Here, we give a proof to the $n=2$ case to show and clarify our algorithm to the readers. Assume that $T_1$ and $T_2$ are doubly commuting $\mathbb{A}_r$-contractions acting on $\mathcal{H}$. It follows from Theorem \ref{canonical decomposition} that there is a unique orthogonal decomposition of $\mathcal{H}$ into closed subspaces $\mathcal{H}_a, \mathcal{H}_c$ which reduce $T_1$ such that $T_1|_{\mathcal{H}_a}$ is an $\A_r$-unitary and $T_1|_{\mathcal{H}_c}$ is a c.n.u. $\A_r$-contraction. Note that $T_1$ is an $\mathbb{A}_r$-contraction which doubly commutes with $T_2$ and so, by Lemma \ref{reducing}, $\mathcal{H}_a$ and $\mathcal{H}_c$ reduce $T_2$. Since $T_2|_{\mathcal{H}_{a}}$ is again an $\mathbb{A}_r$-contraction, by Theorem \ref{canonical decomposition}, there is a unique canonical decomposition $\mathcal{H}_{a}=\mathcal{H}_{aa} \oplus \mathcal{H}_{ac}$ such that $\mathcal{H}_{aa}, \mathcal{H}_{ac}$ are closed subspaces reducing $T_2$ to an $\mathbb{A}_r$-unitary and a c.n.u. $\mathbb{A}_r$-contraction respectively. Again, it follows from  Lemma \ref{reducing} that $\mathcal{H}_{aa}, \mathcal{H}_{ac}$ reduce $T_1$. As $T_2|_{\mathcal{H}_{c}}$ is also an $\mathbb{A}_r$-contraction, it follows from Theorem \ref{canonical decomposition} that there is a unique canonical decomposition of $\mathcal{H}_c$ into closed subspaces $\mathcal{H}_{ca}, \mathcal{H}_{cc}$ reducing $T_2$ such that $T_2|_{\mathcal{H}_{ca}}$ is an $\mathbb{A}_r$-unitary and $T_2|_{\mathcal{H}_{cc}}$ is a c.n.u. $\mathbb{A}_r$-contraction. An application of Lemma \ref{reducing} yields that $\mathcal{H}_{ca}, \mathcal{H}_{cc}$ reduce $T_1$. Consequently, we have the following common reducing subspaces of $T_1, T_2$: $
			\mathcal{H}=\mathcal{H}_{aa} \oplus \mathcal{H}_{ac} \oplus \mathcal{H}_{ca} \oplus  \mathcal{H}_{cc}
		$
		such that 
		\begin{enumerate}
			\item $T_1|_{\mathcal{H}_{aa}}, T_1|_{\mathcal{H}_{ac}}, T_2|_{\mathcal{H}_{aa}}$ and $T_2|_{\mathcal{H}_{ca}}$ are $\mathbb{A}_r$-unitaries; 
			\item $T_1|_{\mathcal{H}_{ca}}, T_1|_{\mathcal{H}_{cc}}, T_2|_{\mathcal{H}_{ac}}$ and $T_2|_{\mathcal{H}_{cc}}$ are c.n.u. $\mathbb{A}_r$-contractions.
		\end{enumerate}
Let $\mathcal{H}_{aa}'$ be the largest closed reducing subspace for $T_1, T_2$ on which each $T_j$ acts as an $\A_r$-unitary. Then $\mathcal{H}_{aa}\subseteq \mathcal{H}_{aa}'$. It follows from Theorem \ref{canonical decomposition} that the canonical decomposition for $T_1$ is given by $\mathcal{H}=\mathcal{H}_a^{(1)} \oplus \mathcal{H}_c^{(1)}$, where 
$
\HS_a^{(1)}=\HS_{aa} \oplus \HS_{ac}$ and $ \HS_c^{(2)}=\HS_{ca} \oplus \HS_{cc}$. Since $T_1|_{\mathcal{H}_{aa}'}$ is an $\A_r$-unitary, then Theorem \ref{canonical decomposition} yields that $\HS_{aa}' \subseteq \HS_{a}^{(1)}=\HS_{aa}\oplus \HS_{ac}$. Similarly, using the canonical decomposition for $\A_r$-contraction $T_2$, one can dedcuce that $\HS_{aa}' \subseteq \HS_{aa}\oplus \HS_{ca}$. Hence, $\HS_{aa}' \subseteq (\HS_{aa} \oplus \HS_{ac}) \cap (\HS_{aa} \oplus \HS_{ca})$ and so, $\HS_{aa}' \subseteq \HS_{aa}$. Therefore, $\HS_{aa}$ is indeed the maximal such subspace. Similarly, one can prove the desired maximality for $\HS_{ac}, \HS_{ca}$ and $\HS_{cc}$. Suppose we have another  decomposition of $\HS$ given by
	$
 \mathcal{H}=\mathcal{H}'_{aa}\oplus \mathcal{H}'_{ac}\oplus \mathcal{H}'_{ca}\oplus \mathcal{H}'_{cc}
	$
	such that 
		\begin{enumerate}
		\item $T_1|_{\mathcal{H}_{aa}'}, T_1|_{\mathcal{H}_{ac}'}, T_2|_{\mathcal{H}_{aa}'}$ and $T_2|_{\mathcal{H}_{ca}'}$ are $\mathbb{A}_r$-unitaries; 
		\item $T_1|_{\mathcal{H}_{ca}'}, T_1|_{\mathcal{H}_{cc}'}, T_2|_{\mathcal{H}_{ac}'}$ and $T_2|_{\mathcal{H}_{cc}'}$ are c.n.u. $\mathbb{A}_r$-contractions.
	\end{enumerate}
It follows from Theorem \ref{canonical decomposition} and the uniqueness of the decomposition for $T_1$ that
	\[
	\mathcal{H}_{aa}\oplus \mathcal{H}_{ac}=\mathcal{H}'_{aa}\oplus \mathcal{H}'_{ac} \quad \mbox{and} \quad \mathcal{H}_{ca}\oplus \mathcal{H}_{cc}=\mathcal{H}'_{ca}\oplus \mathcal{H}'_{cc}.
	\]
	Again by Theorem \ref{canonical decomposition} and the uniqueness of the decomposition for $T_2$, we have
	\[
	\mathcal{H}_{aa}\oplus \mathcal{H}_{ca}=\mathcal{H}'_{aa}\oplus \mathcal{H}'_{ca} \quad \mbox{and} \quad  \mathcal{H}_{ac}\oplus \mathcal{H}_{cc}=\mathcal{H}'_{ac}\oplus \mathcal{H}'_{cc}.
	\]
	Combining these, we get $ \mathcal{H}_{aa}\oplus \mathcal{H}_{ac} \oplus \mathcal{H}_{ca} =\mathcal{H}'_{aa}\oplus \mathcal{H}'_{ac} \oplus \mathcal{H}'_{ca}$.
	Taking orthogonal complements we have that $\mathcal{H}_{cc}=\mathcal{H}'_{cc}$. Similarly, we can prove that $\mathcal{H}_{aa}=\mathcal{H}'_{aa}, \mathcal{H}_{ac}=\mathcal{H}'_{ac}$ and $\mathcal{H}_{ca}=\mathcal{H}'_{ca}$. Hence, the desired decomposition holds for $n=2$ and is unique. Assume that the statement holds for all $n$ up to $k$. We now prove the $n=k+1$ case. Let $T_1, \dotsc, T_k, T_{k+1}$ be doubly commuting $\A_r$-contractions on $\mathcal{H}$. Consider the collection of operator tuples 
	\[
	\{(T_1|_{\mathcal{H}_j}, \dotsc, T_k|_{\mathcal{H}_j}, \; T_{k+1}|_{\mathcal{H}_j}) \ : \  1 \leq j \leq 2^k\}
	\]
	and take any member, say
	$
	(T_1|_{\mathcal{H}_j}, \dotsc, T_k|_{\mathcal{H}_j}, \; T_{k+1}|_{\mathcal{H}_j})
	$ 
 from this collection. By induction hypothesis, each $T_i|_{\mathcal{H}_j} (1 \leq i \leq k)$ is either an $\A_r$-unitary or a c.n.u. $\A_r$-contraction. By Lemma \ref{reducing}, each $\HS_j$ reduces $T_{k+1}$. We apply Theorem \ref{canonical decomposition} on $T_{k+1}|_{\mathcal{H}_j}$. Consequently, $\mathcal{H}_j$ can be written as an orthogonal sum of two reducing subspaces for $T_{k+1}$, say, $\mathcal{H}_{ja}, \mathcal{H}_{jc}$ such that $T_{k+1}|_{\mathcal{H}_{ja}}$ is an $\A_r$-unitary and $T_{k+1}|_{\mathcal{H}_{jc}}$ is a c.n.u. $\A_r$-contraction. Again, by Lemma \ref{reducing}, the subspaces $\mathcal{H}_{ja}, \mathcal{H}_{jc}$ reduce each of $T_1|_{\mathcal{H}_{j}}, \dotsc, T_k|_{\mathcal{H}_{j}}$. Hence, each tuple 
	$
	(T_1|_{\mathcal{H}_j}, \dotsc, T_k|_{\mathcal{H}_j},  T_{k+1}|_{\mathcal{H}_j})
	$
	orthogonally decomposes into two parts. Continuing this way, for every tuple in the collection, we get $2^{k+1}$ orthogonal subspaces and the other parts of the theorem follow from the induction hypothesis. The proof is now complete.
	\end{proof}
	Needless to mention that the canonical decomposition of an $\A_r$-isometry $V$ is nothing but the Wold decomposition of $V$. Therefore, as a particular case of the canonical decomposition for doubly commuting $\A_r$-contractions described above, we have the following Wold-type decomposition for a finite family of doubly commuting $\mathbb{A}_r$-isometries. Naturally, we also skip the proof.
	\begin{thm}\label{prop603}
Assume that $(V_1, \dotsc, V_n)$ is a tuple of doubly commuting $\mathbb{A}_r$-isometries acting on a Hilbert space $\mathcal{H}$. Then there exists a unique decomposition of $\HS$ into an orthogonal sum of $2^n$ subspaces $\HS_1, \dotsc, \HS_{2^n}$ of $\HS$ such that
	\begin{enumerate}
		\item every $\HS_j
		, 1 \leq j \leq 2^n$, is a joint reducing subspace for $V_1, \dotsc, V_n$ ;
		\item $V_i|_{\HS_j}$ is either an $\A_r$-unitary or a pure $\A_r$-isometry  for $ 1 \leq i \leq n$ and $1 \leq j \leq 2^n$;
		\item if $\Phi_n$ is the collection of all functions $\phi:\{1, \dotsc, n\} \to \{t_u, t_c\}$, then for every $\phi \in \Phi_n$, there is a unique maximal joint reducing subspace $\HS_\ell (1 \leq \ell \leq 2^n)$ such that the $j$-th component of $(V_1|_{\HS_{\ell}}, \dotsc, V_n|_{\HS_{\ell}})$ is of the type $\phi(j)$ for $1 \leq j \leq n$;
		\item $\HS_1$ is the maximal joint reducing subspace for $V_1, \dotsc, V_n$ such that $V_1|_{\HS_1}, \dotsc, V_n|_{\HS_1}$ are $\A_r$-unitaries;
		\item $\HS_{2^n}$ is the maximal joint reducing subspace for $V_1, \dotsc, V_n$ such that $V_1|_{\HS_{2^n}}, \dotsc, V_n|_{\HS_{2^n}}$ are pure $\A_r$-isometries.
	\end{enumerate}
One or members of the subspaces in the decomposition may be $\{0\}$.
	\end{thm}
In Theorem {Levan type decomposition}, we have seen that a c.n.u. $\A_r$-contraction further admits an orthogonal decomposition into an $\A_r$-isometry and a c.n.i. $\A_r$-contraction. This is Levan-type decomposition for a c.n.u. $\A_r$-contraction. We now find a generalization of this result for a finite family of doubly commuting $\A_r$-contractions. Indeed, we show that the tuple $(T_1|_{\mathcal{H}_{2^n}}, \dotsc, T_n|_{\mathcal{H}_{2^n}})$ as in Theorem \ref{thm602}, can further be decomposed into $2^n$ orthogonal parts by applying an argument similar to that in the proof of Theorem \ref{thm602}. We obtain this result by an application of the following lemma. 
\begin{lem}\label{Levan_reduce}
	Let $T_1, T_2$ be doubly commuting $\A_r$-contractions acting on a Hilbert space $\mathcal{H}$ and let $T_1$ be a c.n.u. $\A_r$-contraction. If $\mathcal{H}=\mathcal{H}_1 \oplus \mathcal{H}_2$ is the orthogonal decomposition for $T_1$ as in Theorem \ref{Levan type decomposition}, then $\mathcal{H}_1, \mathcal{H}_2$ reduce $T_2$. 
\end{lem}

\begin{proof}
Let $\Lambda$ be the collection of all closed subspaces of $\mathcal{H}$  that reduces $T_1$ to an $\A_r$-isometry. It follows from the proof of Theorem \ref{Levan type decomposition} that $\mathcal{H}_1$ is the closure of span of $\Lambda$. For any subspace $\mathcal{L} \in \Lambda$, we show that $\overline{T_2\mathcal{L}} \subseteq \HS_1$. Note that 
$
T_1(T_2\mathcal{L})=T_2(T_1\mathcal{L}) \subseteq T_2\mathcal{L}$ and $  T_1^*(T_2\mathcal{L})=T_2(T_1^*\mathcal{L}) \subseteq T_2\mathcal{L},$
where we used the facts that $\mathcal{L}$ is a reducing subspace for $T_1$ and $T_1^*T_2=T_2T_1^*$. By continuity, $\overline{T_2\mathcal{L}}$ is a reducing subspace for $T_1$. For any $x \in \mathcal{L}$, it follows from Theorem \ref{thm407} that 
$
-T_1^{*2}T_1^2x+(1+r^2)T_1^*T_1x-r^2x=0$ and so, $-T_1^{*2}T_1^2(T_2x)+(1+r^2)T_1^*T_1(T_2x)-r^2T_2x=0$.
Again by continuity, $-T_1^{*2}T_1^2y+(1+r^2)T_1^*T_1y-r^2y=0$ for all $y \in \overline{T_2\mathcal{L}}$. Consequently, Theorem \ref{thm407} yields that $T_1|_{\overline{T_2\mathcal{L}}}$ is an $\A_r$-isometry. Again, it follows from the proof of Theorem \ref{Levan type decomposition} that $\mathcal{H}_1$ is the largest closed reducing subspace of $\mathcal{H}$ which reduces $T_1$ to an $\A_r$-isometry and so, $\overline{T_2\mathcal{L}} \subseteq \mathcal{H}_1$. Since $\mathcal{L} \in \Lambda$ is arbitrary, we get $T_2(\text{span} \ \Lambda) \subseteq \mathcal{H}_1$. Using continuity arguments, one can show that $T_2\mathcal{H}_1 \subseteq \mathcal{H}_1$. Similarly, it can be proved that $T_2^*\mathcal{H}_1 \subseteq \mathcal{H}_1$. Therefore, $\mathcal{H}_1$ reduces $T_2$ and then, so does $\mathcal{H}_2$. The proof is now complete. 
\end{proof}
We shall use the following terminologies in the rest of the paper.	

\begin{defn} \label{defn:063}
	Let $T$ be an $\A_r$-contraction acting on a Hilbert space $\HS$. Then
	\begin{enumerate}
	\item T is a \textit{fundamental c.n.u. $\A_r$-contraction} if either $T$ is a pure $\A_r$-isometry or a c.n.i. $\A_r$-contraction;
	
		\item[(i)] $T$ is a \textit{fundamental c.n.u. $\A_r$-contraction} of \textit{type $t_p$} if $T$ is a pure $\A_r$-isometry;
		
		\item[(iii)] $T$ is a \textit{fundamental c.n.u. $\A_r$-contraction} of \textit{type $t_{cni}$} if $T$ is a c.n.i. $\A_r$-contraction;
		
		\item $T$ is a \textit{non-fundamental c.n.u. $\A_r$-contraction} if $T$ is a c.n.u. $\A_r$-contraction which is neither a pure $\A_r$-isometry nor a c.n.i. $\A_r$-contraction.
		
	\end{enumerate}	
\end{defn}

We now reach another main result of this Section, a canonical decomposition of a finite family of doubly commuting $\A_r$-contractions.

	\begin{thm}\label{thm605}
		Let $(T_1, \dotsc, T_n)$ be a doubly commuting tuple of c.n.u. $\mathbb{A}_r$-contractions acting on a Hilbert space $\mathcal{H}$. Then there exists a unique decomposition of $\HS$ into an orthogonal sum of $2^n$ subspaces $\HS_1, \dotsc, \HS_{2^n}$ of $\HS$ such that
\begin{enumerate}
	\item every $\HS_j
	, 1 \leq j \leq 2^n$, is a joint reducing subspace for $T_1, \dotsc, T_n$ ;
	\item $T_i|_{\HS_j}$ is either a pure $\A_r$-isometry or a c.n.i. $\A_r$-contraction for $ 1 \leq i \leq n$ and $1 \leq j \leq 2^n$;
	\item if $\mathscr{X}_n$ is the collection of all functions $\chi:\{1, \dotsc, n\} \to \{t_p, t_{cni}\}$, then for every $\chi \in \mathscr{X}_n$, there is a unique maximal joint reducing subspace $\HS_m (1 \leq m \leq 2^n)$ such that the $j$-th component of $(T_1|_{\HS_m}, \dotsc, T_n|_{\HS_m})$ is of the type $\chi(j)$ for $1 \leq j \leq n$.
\end{enumerate}
 One or more subspaces in the decomposition  may coincide with the subspace $\{0\}$.
\end{thm}		
\begin{proof}
	We apply the same algorithm as in the proof of Theorem \ref{thm602} and apply Lemma \ref{Levan_reduce} repeatedly to reach the desired decomposition.
\end{proof}	

\vspace{0.2cm}
	\section{A different canonical decomposition for commuting $\A_r$-contractions}\label{commuting}
	
	\vspace{0.2cm}
	
	\noindent As we have mentioned before that we do not expect to have orthogonal decompositions as in Theorems \ref{thm602} \& \ref{prop603} under commutativity assumption only. Indeed, we have seen that double commutativity plays a crucial role in determining those common reducing subspaces. In this Section, we shall obtain a different type of canonical decomposition for a finite family of commuting $\A_r$-contractions. We are motivated by the fine work of Burdak \cite{BurdakIII}, where he removed the double commutativity assumption on a finite family of commuting contractions and determined their explicit structure by exhibiting a certain canonical decomposition. We explain the contrast taking two contractions at a time. For a doubly commuting pair of contractions $(T_1, T_2)$ acting on a Hilbert space $\HS$, there is a unique decomposition of $\mathcal{H}$ into an orthogonal sum of subspaces $\HS_{uu}, \HS_{uc}, \HS_{cu}, \HS_{cc}$ reducing $T_1, T_2$ 	such that 
	\begin{enumerate}
		\item $T_1|_{\mathcal{H}_{uu}},  T_2|_{\mathcal{H}_{uu}}$ are unitaries; 
		
		\item $T_1|_{\mathcal{H}_{uc}}$ is a unitary and $T_2|_{\mathcal{H}_{uc}}$ is a c.n.u. contraction;
		 \item $T_1|_{\mathcal{H}_{cu}}$ is a c.n.u. contraction and $T_2|_{\mathcal{H}_{cu}}$ is a unitary;
		\item $ T_1|_{\mathcal{H}_{cc}} $ and $T_2|_{\mathcal{H}_{cc}}$ are c.n.u. contractions.
	\end{enumerate} 
An interested reader is referred to refer \cite{Slocinski, Slocinski-I} for further details. On the other hand, Burdak \cite{BurdakIII} assumed only commutativity of $T_1, T_2$ and obtained closed joint reducing subspaces $\HS_{uu}, \HS_{uc}, \HS_{cu}, \HS'$ of $T_1, T_2$ such that $(1),\,(2)$ and $(3)$ out of the above four conclusions hold.
However, $T_1, T_2$ need not be c.n.u. contractions on the orthogonal complement of $\HS_{uu} \oplus \HS_{uc} \oplus \HS_{cu}$. See Theorem 2.1 in \cite{BurdakIII} for a proof to this result. It turns out that $\HS'=\HS \ominus (\HS_{uu} \oplus \HS_{uc} \oplus \HS_{cu})$ is a joint reducing subspace for $T_1,T_2$ such that if $\mathcal{L}$ is a proper closed subspace of $\HS'$ reducing $T_1, T_2$ and either of $T_1|_\mathcal{L}$ or $T_2|_\mathcal{L}$ is a unitary, then $\mathcal{L}=\{0\}$.  Burdak called such a pair a \textit{strongly completely nonunitary (or simply strongly c.n.u.) pair}. Taking cue from this work, we find a similar orthogonal decomposition for commuting $\A_r$-contractions. We begin with the definitions of analogues of these special classes for $\A_r$-contractions.

	\begin{defn}
		A tuple $(T_1, \dotsc, T_n)$ of commuting $\mathbb{A}_r$-contractions acting on Hilbert space $\mathcal{H}$ is said to be a 
		\begin{enumerate}
			\item \textit{strongly c.n.u. tuple of $\mathbb{A}_r$-contractions} if there is no non-zero closed subspace $\mathcal{L}$ of $\mathcal{H}$ reducing each $T_j$ such that at least $(n-1)$ among $T_1|_\mathcal{L}, \dotsc, T_n|_\mathcal{L}$ are $\mathbb{A}_r$-unitaries.
			
			\smallskip
			
			\item \textit{strongly pure $\mathbb{A}_r$-isometric tuple} if each $T_j$ is an $\A_r$-isometry and there is no non-zero closed subspace $\mathcal{L}$ of $\mathcal{H}$ reducing each $T_j$ such that at least $(n-1)$ among $T_1|_\mathcal{L}, \dotsc, T_n|_\mathcal{L}$ are $\mathbb{A}_r$-unitaries.
		\end{enumerate}
		
	\end{defn}
	
\begin{thm}\label{thm702}
Let $(T_1, \dotsc, T_n)$ be a tuple of commuting  $\A_r$-contractions acting on a Hilbert space $\mathcal{H}$. Then there corresponds a unique decomposition of $\mathcal{H}$ into an orthogonal sum of $(n+2)$ subspaces $\HS_1, \dotsc, \HS_{n+1}, \HS(s)$ of $\HS$ such that
\begin{enumerate}
\item $\HS_1, \dotsc, \HS_{n+1}, \HS(s)$ are joint reducing subspaces for $T_1, \dotsc, T_n$;
\item $T_i|_{\HS_j}$ is either an $\A_r$-unitary or a c.n.u. $\A_r$-contraction for $ 1 \leq i \leq n$ and $1 \leq j \leq n+1$;
\item If $\widetilde{\Phi_n}$ is the collection of all functions $\phi: \{1, \dotsc, n\} \to \{t_u, t_c\}$ such that there are at least $(n-1)$ values as $t_u$ among $\phi(1), \dotsc, \phi(n)$ , then corresponding to every $\phi \in \widetilde{\Phi_n}$, there is a unique maximal joint reducing subspace $\HS_\ell ( 1 \leq \ell \leq n+1)$ such that the $j$-th component of $(T_1|_{\HS_{\ell}}, \dotsc, T_n|_{\HS_{\ell}})$ is of the type $\phi(j)$ for $1 \leq j \leq n$;
\item $(T_1|_{\mathcal{H}(s)}, \dotsc, T_n|_{\mathcal{H}(s)})$ is a strongly c.n.u. tuple of $\A_r$-contractions.
\end{enumerate}
One or more subspaces in the decomposition  may coincide with the subspace $\{0\}$.
\end{thm}
\begin{proof}
Let us consider the space
\[
\mathcal{H}^{dc}=\overline{\mbox{span}}\bigg\{\mathcal{H}_0 \subseteq \underset{1 \leq i \ne j \leq n}{\bigcap} \ \mbox{Ker}(T_iT_j^*-T_jT_i^*):  P_{\mathcal{H}_0}T_i=T_iP_{\mathcal{H}_0}, 1 \leq i \leq n \bigg\}.
\]
Evidently, $\HS^{dc}$ is the largest joint reducing subspace for $T_1, \dotsc, T_n$ such that
$
(T_1|_{\HS^{dc}}, \dotsc, T_n|_{\HS^{dc}})
$
becomes a tuple of doubly commuting $\A_r$-contractions. It follows from Theorem \ref{thm602} that $\mathcal{H}^{dc}$ admits a unique decomposition into an orthogonal sum of $2^n$ joint reducing subspaces of $T_1, \dots , T_n$, say $\HS^{dc}_1, \dotsc, \HS^{dc}_{2^n}$ such that
\begin{enumerate}
\item $T_i|_{\HS^{dc}_j}$ is either an $\A_r$-unitary or a c.n.u. $\A_r$-contraction for $ 1 \leq i \leq n$ and $1 \leq j \leq 2^n$;
\item if $\Phi_n$ is the collection of all functions $\phi:\{1, \dotsc, n\} \to \{t_u, t_c\}$, then for every $\phi \in \Phi_n$, there is a unique maximal joint reducing subspace $\HS^{dc}_\ell (1 \leq \ell \leq 2^n)$ such that the $j$-th component of $(T_1|_{\HS^{dc}_{\ell}}, \dotsc, T_n|_{\HS^{dc}_{\ell}})$ is of the type $\phi(j)$ for $1 \leq j \leq n$;
\item $T_1|_{\HS^{dc}_1}, \dotsc, T_n|_{\HS^{dc}_1}$ are $\A_r$-unitaries and $T_1|_{\HS^{dc}_{2^n}}, \dotsc, T_n|_{\HS^{dc}_{2^n}}$ are c.n.u. $\A_r$-contractions.
\end{enumerate}
Let $\phi_1 \in \Phi_n$ be the constant map $\phi_1(j)=t_u$ for all $j$. We now consider the maps in $\Phi_n$ which take the value $t_c$ at only one point. Indeed, there are precisely $n$ such maps in $\Phi_n$. Let $\phi_2, \dotsc, \phi_{n+1}$ be those $n$ maps. Up to permutation of indices there is a unique maximal joint reducing subspace $\HS_{\ell}^{dc} \subset \HS^{dc}$ of $T_1, \dots , T_n$ for every $\phi_\ell$ for $1 \leq \ell \leq n+1)$ such that the $j$-th component of $(T_1|_{\HS^{dc}_{\ell}}, \dotsc, T_n|_{\HS^{dc}_{\ell}})$ is of the type $\phi_\ell(j)$ for $1 \leq j \leq n$. Let $\HS_\ell$ be the maximal joint reducing subspace for $T_1, \dotsc, T_n$ such that the $j$-th component of $(T_1|_{\HS_{\ell}}, \dotsc, T_n|_{\HS_{\ell}})$ is of the type $\phi_\ell(j)$ for $1 \leq j \leq n$ and $1 \leq \ell \leq n+1$. It follows from the definition of $\phi_\ell$ that at least $(n-1)$ among $T_1|_{\HS_{\ell}}, \dotsc, T_n|_{\HS_{\ell}}$ are $\A_r$-unitaries. Using Fuglede's Theorem \cite{Fuglede}, we have that $(T_1|_{\HS_{\ell}}, \dotsc, T_n|_{\HS_{\ell}})$ is a doubly commuting tuple and thus $\HS_\ell \subseteq \HS^{dc}$. Since $\HS_\ell^{dc} \subseteq \HS^{dc}$ is the largest joint reducing subspace of $T_1, \dots , T_n$ such that each $T_j|_{\HS^{dc}_{\ell}}$ is of the type $\phi(j)$, we have that $\HS_{\ell} \subseteq \HS_{\ell}^{dc}$. Thus, $\HS_{\ell}=\HS_{\ell}^{dc}$ for every $\ell = 1, \dots , n+1$. Consider the following closed joint reducing subspace for $T_1, \dotsc, T_n$:
$
\mathcal{H}(s)=\mathcal{H} \ominus (\HS_1 \oplus \dotsc \oplus \HS_{n+1}).
$
Let if possible there be a closed subspace $\mathcal{L} \subseteq \mathcal{H}(s)$ reducing each $T_j$ such that at least $(n-1)$ among the $\A_r$-contractions $T_1|_{\mathcal{L}}, \dotsc T_n|_{\mathcal{L}}$ are $\mathbb{A}_r$-unitaries. Without loss of generality, let $T_1|_{\mathcal{L}}, \dotsc ,T_{n-1}|_{\mathcal{L}}$ be those $\A_r$-unitaries. Again, by Fuglede's theorem \cite{Fuglede}, $(T_1|_\mathcal{L}, \dotsc T_n|_\mathcal{L})$ is a tuple of doubly commuting $\A_r$-contractions.  By the definition of $\mathcal{H}^{dc}$, we have $\mathcal{L} \subseteq \mathcal{H}^{dc}$. It follows from Theorem \ref{canonical decomposition} that there exists a unique decomposition of $\mathcal{L}$ into an orthogonal sum of two closed subspaces reducing $T_n|_\mathcal{L}$, say $\mathcal{L}= \mathcal{L}_{a} \oplus \mathcal{L}_{c},$ such that the $T_n|_{\mathcal{L}_{a}}$ is an $\mathbb{A}_r$-unitary and $T_n|_{\mathcal{L}_{c}}$ is a c.n.u. $\mathbb{A}_r$-contraction. By Lemma \ref{reducing}, $\mathcal{L}_a, \mathcal{L}_c$ are reducing subspaces for $T_1|_\mathcal{L}, \dotsc, T_{n-1}|_\mathcal{L}$ as well. Since $\mathcal{H}^{dc}_{1}$ is the largest closed subspace of $\mathcal{H}^{dc}$ reducing each $T_j$ to an $\A_r$-unitary, we have $\mathcal{L}_a \subseteq \mathcal{H}^{dc}_{1}$. For some $m$ with $ 2 \leq m \leq n+1$, there is the largest closed subspace $\mathcal{H}^{dc}_{m}$ of $\mathcal{H}^{dc}$ reducing each $T_1, \dotsc, T_{n-1}$ to an $\A_r$-unitary and $T_n$ to a c.n.u. $\A_r$-contraction. So, we have that $\mathcal{L}_c \subseteq \mathcal{H}^{dc}_{m}$. Consequently,
\[
\mathcal{L}=\mathcal{L}_a\oplus \mathcal{L}_c \subseteq \mathcal{H}^{dc}_{1} \oplus \mathcal{H}^{dc}_{m} = \mathcal{H}_{1} \oplus \mathcal{H}_{m} \subseteq \underset{1 \leq \ell \leq n+1}{\bigoplus}\mathcal{H}_{\ell} \,.
\]
Therefore, $\mathcal{L} \subseteq \mathcal{H}(s) \cap (\mathcal{H}\ominus \mathcal{H}(s))=\{0\}$ and so $(T_1|_{\mathcal{H}(s)}, \dotsc, T_n|_{\mathcal{H}(s)})$ is a strongly c.n.u. tuple of $\A_r$-contractions. The uniqueness part follows from the maximality of $\HS_{\ell}$ for $1 \leq \ell \leq n+1$. The proof is now complete. 
 
\end{proof}
 As a particular case of the above result, we obtain a Wold type decomposition for a tuple of commuting $\A_r$-isometries in the same spirit.
\begin{thm}
Let $(V_1, \dotsc, V_n)$ be a tuple of commuting $\A_r$-isometries acting on a Hilbert space $\mathcal{H}$. Then $\mathcal{H}$ admits a unique decomposition into an orthogonal sum of $(n+2)$ closed joint reducing subspaces $\HS_1, \dotsc, \HS_{n+1}, \HS(s)$ for $V_1, \dotsc, V_n$ such that
\begin{enumerate}
\item $V_i|_{\HS_j}$ is either an $\A_r$-unitary or a pure $\A_r$-isometry for $ 1 \leq i \leq n$ and $1 \leq j \leq n+1$;
\item If $\widetilde{\Phi_n}$ is the collection of all functions $\phi: \{1, \dotsc, n\} \to \{t_u, t_p\}$ such that there are at least $(n-1)$ values as $t_u$ among $\phi(1), \dotsc, \phi(n)$ , then corresponding to every $\phi \in \widetilde{\Phi_n}$, there is a unique maximal joint reducing subspace $\HS_\ell ( 1 \leq \ell \leq n+1)$ such that the $j$-th component of $(V_1|_{\HS_{\ell}}, \dotsc, V_n|_{\HS_{\ell}})$ is of the type $\phi(j)$ for $1 \leq j \leq n$;
\item $(V_1|_{\mathcal{H}(s)}, \dotsc, V_n|_{\mathcal{H}(s)})$ is a strongly pure $\A_r$-isometric tuple.
\end{enumerate}
One or more subspaces in the decomposition may coincide with the subspace $\{0\}$.

\end{thm}	
	
Since all the way we have discussed about commuting and doubly commuting $\A_r$-contractions in this and the previous Sections, we want to warrant the existence of pair of commuting $\A_r$-contractions that are not doubly commuting. More precisely, we shall construct a pair of c.n.u. $\A_r$-contractions for this.
	
		\begin{eg}
		Consider the space
					$L^2(\partial\mathbb{A}_r)=L^2(\mathbb{T})\oplus L^2(r\mathbb{T})$,
				where $L^2(\mathbb{T})$ and $L^2(r\mathbb{T})$ are endowed with normalized Lebesgue measure and the norm is given by
		\begin{equation*}
			\|f\|^2=\frac{1}{2\pi}\overset{2\pi}{\underset{0}{\int}}|f(e^{it})|^2dt+\frac{1}{2\pi}\overset{2\pi}{\underset{0}{\int}}|f(re^{it})|^2dt  .  
		\end{equation*}
		Define the operator $S$ on $L^2(\partial\mathbb{A}_r)$ by
					$S(f)(z)=zf(z)$.
		 We recall from Example \ref{eg2} that the set $\{w_{\alpha+n}: n \in \mathbb{Z}\}$ is an orthogonal set in $L^2(\partial \mathbb{A}_r)$. The space $H_{\alpha}^2(\partial \mathbb{A}_r)$ is the closure of the span in $L^2(\partial \mathbb{A}_r)$ of the functions $w_{\alpha +n}$. It follows from Example \ref{eg2} that $S_\alpha=S|_{H_{\alpha}^2(\partial \mathbb{A}_r)}$ is a pure $\mathbb{A}_r$-isometry and that $\sigma(S_\alpha)=\overline{\mathbb{A}}_r$. Since $S_\alpha$ has $\overline{\mathbb{A}}_r$ as a spectral set and any compact superset containing a spectral set is again a spectral set, it follows that $\CA_{r^2}$ is a spectral set for $S_\alpha$. Hence, $S_\alpha$ is an $\A_{r^2}$-contraction. We now show that $S_\alpha$ is actually a c.n.u. $\mathbb{A}_{r^2}$-contraction on $H_{\alpha}^2(\partial \mathbb{A}_r)$.		
		
		\smallskip
		
		 Assume that $\mathcal{H} \subseteq H_{\alpha}^2(\partial \mathbb{A}_r)$ is a closed subspace reducing $S_\alpha$ to an $\mathbb{A}_{r^2}$-unitary. By virtue of Theorem \ref{thm302}, we may decompose $\mathcal{H}$ into closed reducing subspaces $\HS_u, \HS_{r^2}$ of $S_\alpha$ as
			$
				\mathcal{H}=\mathcal{H}_u \oplus \mathcal{H}_{r^2}
			$ 
			such that $S_\alpha|_{\mathcal{H}_u}$ and  $r^2S^{-1}_\alpha|_{\mathcal{H}_{r^2}}$ are unitaries. Since $S_\alpha$ is a pure $\mathbb{A}_r$-isometry, it has no $\A_r$-unitary part and so, $\mathcal{H}_u=\{0\}$. Note that $S_\alpha|_{\mathcal{H}_{r^2}}$ is $\mathbb{A}_r$-isometry because $\mathcal{H}_{r^2}$ is a closed subspace of $H_{\alpha}^2(\partial \mathbb{A}_r)$ that reduces $S_\alpha$. It follows from Theorem \ref{thm404} that either $\sigma(S_\alpha|_{\mathcal{H}_{r^2}})=\overline{\mathbb{A}}_r $ or $\sigma(S_\alpha|_{\mathcal{H}_{r^2}})\subseteq \partial\mathbb{A}_r$.
Since $r^2S^{-1}_\alpha|_{\mathcal{H}_{r^2}}$ is a unitary, $ \sigma(S_\alpha|_{\mathcal{H}_{r^2}}) \subseteq r^2\mathbb{T}$ which is a contradiction. Thus, $\mathcal{H}_{r^2}=\{0\}$ and so, $\HS=\{0\}$. Hence, $S_\alpha$ is a c.n.u. $\mathbb{A}_{r^2}$-contraction.
		By spectral mapping theorem, we have  
		\[
		\sigma(S_\alpha^2)=\{\lambda^2: \lambda \in \sigma(S_\alpha) \}=\CA_{r^2}.
		\]
	 Let $f$ be a rational function with poles off $\CA_{r^2}$ and let $g(z)=f(z^2)$. Then $g$ is a rational function with poles off $\CA_r$. Using the fact that $S_\alpha$ is an $\A_r$-contraction, we have
		\begin{equation*}
			\begin{split}
				\|f(S_\alpha^2)\|  = \|g(S_\alpha)\| 
				 \leq \max \{|g(z)|: z \in  \CA_r\}  \leq \max \{|f(z^2)| : z \in \CA_r \} \leq \|f\|_{\infty, \CA_{r^2}}. \\
			\end{split}
		\end{equation*}
		Hence, $S_\alpha^2$ is an $\mathbb{A}_{r^2}$-contraction. Therefore, $(S_\alpha, S_\alpha^2)$ is a commuting pair of $\A_{r^2}$-contractions on $H_{\alpha}^2(\partial \mathbb{A}_r)$. Now we show that $S_\alpha^2$ is a c.n.u. $\mathbb{A}_{r^2}$-contraction.
		
		\smallskip
	
We use the similar computations as in Example \ref{eg2}. Take any $f$ in $H_{\alpha}^2(\partial \mathbb{A}_r)$. Then there is some $\{f_n\}_{n \in \Z} \subset \C$ such that $	f=\underset{n \in \Z}{\sum}f_nw_{\alpha+n}$. Consequently,
\[
\|f\|^2-\|S_\alpha^2f\|^2=\underset{n \in \Z}{\sum}(\|f_nw_{\alpha+n}\|^2-\|f_nw_{\alpha+n+2}\|^2)
=\underset{n \in \Z}{\sum}|f_n|^2(\|w_{\alpha+n}\|^2-\|w_{\alpha+n+2}\|^2)
\]
and 
\[
\|f\|^2-\|r^2(S_\alpha^2)^{-1}f\|^2=\underset{n \in \Z}{\sum}(\|f_nw_{\alpha+n}\|^2-\|r^2f_nw_{\alpha+n-2}\|^2)
=\underset{n \in \Z}{\sum}|f_n|^2(\|w_{\alpha+n}\|^2-r^4\|w_{\alpha+n-2}\|^2).
\]
Since 
\[
\|w_{\alpha+n}\|^2-\|w_{\alpha+n+2}\|^2=r^{2(\alpha+n)}(1-r^4) \quad \text{and} \quad \|w_{\alpha+n}\|^2-r^4\|w_{\alpha+n-2}\|^2=(1-r^4),
\]
we must have $\|S_\alpha^2f\|=\|f\|$ or $\|r^2S_\alpha^{-2}f\|=\|f\|$ if and only if $f=0$.	Thus, neither $S_\alpha^2$ contains a unitary part nor it contains a $r^2$-times unitary part. Therefore, $S_\alpha^2$ does not contain any $\A_r$-unitary part and thus is a c.n.u. $\A_{r^2}$-contraction.

\smallskip

It remains to show that $S_\alpha$ and $S_\alpha^2$ do not doubly commute. For the ease of computations, let $(V_1, V_2)=(S_\alpha, S_\alpha^2)$. Note that
		 $\langle V_1^*(w_{\alpha+n}), w_{\alpha+m}\rangle =\langle w_{\alpha+n}, V_1w_{\alpha+m}\rangle
			=\langle w_{\alpha+n}, w_{\alpha+m+1}\rangle$. Hence,
			\[ 
			\langle V_1^*(w_{\alpha+n}), w_{\alpha+m}\rangle= \left\{
			\begin{array}{ll}
				0 & n \ne m+1 \\
				1+r^{2(\alpha+n)} & n=m+1 \; ,\\
			\end{array} 
			\right. 
			\]    
			which implies that $V_1^*(w_{\alpha+n})$ is orthogonal to $\overline{\text{span}}\{w_{\alpha+m} : m\ne n+1 \in \mathbb{Z}\}$. Therefore, $V_1^*(w_{\alpha+n})$ is in $\mbox{span}\{w_{\alpha+n-1}\}$ and so,   $V_1^*(w_{\alpha+n})=r(\alpha;n)w_{\alpha+n-1}$ for some $r(\alpha;n) \in \C$. It follows from Example \ref{eg2} that
		\[
					1+r^{2(\alpha+n)}= \langle V_1^*(w_{\alpha+n}), w_{\alpha+n-1}\rangle =\langle r(\alpha;n)w_{\alpha+n-1}, w_{\alpha+n-1}\rangle
					=r(\alpha;n)(1+r^{2(\alpha+n-1)}).
		\]	
Therefore, $r(\alpha;n)=(1+r^{2(\alpha+n)})(1+r^{2(\alpha+n-1)})^{-1}$. Since $V_2=V_1^2$, we have that
			\[
					V_2^*(w_{\alpha+n})
					=V_1^*V_1^*(w_{\alpha+n})
					=V_1^*(r(\alpha;n) w_{\alpha+n-1})
				=r(\alpha;n)r(\alpha;n-1)w_{\alpha+n-2}.
					\]
Putting everything together, we have
			\[
		V_1^*(w_{\alpha+n})=r(\alpha;n)w_{\alpha+n-1}, \quad  \text{and} \quad 	V_2^*(w_{\alpha+n})=r(\alpha;n)r(\alpha;n-1)w_{\alpha+n-2} \; ,
		\]
where
\[ 
r(\alpha;n)=\frac{1+r^{2(\alpha+n)}}{1+r^{2(\alpha+n-1)}} \qquad (n \in \Z).
			\]
A routine calculation yields 
\[
 V_1V_2^*(w_\alpha)=r(\alpha;0)r(\alpha;-1)w_{\alpha-1} \quad \text{and} \quad V_2^*V_1(w_\alpha)=r(\alpha;1)r(\alpha;0)w_{\alpha-1}.
 \] 
So, we have $V_1V_2^*(w_\alpha) \ne V_2^*V_1(w_\alpha)$ as $r(\alpha;-1) \ne r(\alpha;1)$. Hence, $V_1, V_2$ are commuting c.n.u. $\mathbb{A}_{r^2}$- contractions which do not doubly commute.
	\end{eg}
	
	\vspace{0.2cm}
	\section{Generalization}\label{conclusion}
\vspace{0.2cm}	
	
\noindent In this Section, we generalize the results of Section \ref{doubly commuting} to any infinite family of doubly commuting $\A_r$-contractions. To do this, we shall follow the combinatorial techniques and algorithm developed in Section 6 of \cite{Pal}. Recall from Definition \ref{defn:061} that an `atom' is either an $\A_r$-unitary or a c.n.u. $\A_r$-contraction. 
 
\subsection{Decomposition of an infinite family of doubly commuting $\A_r$-contractions} \label{Subsec:final-01}

Let us consider a set of doubly commuting $\A_r$-contractions $\mathcal{T}=\{T_\gamma \ : \ \gamma \in \Omega\}$, where $\Omega$ is an infinite set and each $T_\gamma$ acts on a Hilbert space $\HS$. We assume that there are infinitely many non-atoms in $\mathcal{T}$ to avoid the triviality. We choose an arbitrary $\gamma \in \Omega$. If $T_{\gamma a} \oplus T_{\gamma c}$ is the canonical decomposition of $T_\gamma$ with respect to $\mathcal{H}=\HS_a \oplus \HS_c$ as in Theorem \ref{canonical decomposition}, then $\mathcal{H}_a, \mathcal{H}_c$ are joint reducing subspaces for $\mathcal{T}$, and thus $\mathcal{T}$ is orthogonally decomposed into two tuples, say 
\[
\mathcal{T}_a=\{T_\mu|_{\mathcal{H}_a}  \ : \mu \in \Omega \} \quad \text{and} \quad \mathcal{T}_c=\{T_\mu|_{\mathcal{H}_c}  \ : \mu \in \Omega \}
\]
and clearly the $\gamma$-th entities ($T_{\gamma a}$ and $T_{\gamma c}$, respectively) are atoms. Let 
\[
\mathbb{F}_0=\{(\mathcal{T}, \HS)\} \quad \text{and} \quad \mathbb{F}_\gamma=\{(\mathcal{T}_a, \HS_a), (\mathcal{T}_c, \HS_c)\}.
\]
Next, we apply the axiom of choice and choose $\beta \in \Omega \setminus \{\gamma\}$. Consider the $\beta$-th components of $\mathcal{T}_a$ and $\mathcal{T}_c$, i.e. $T_\beta|_{\mathcal{H}_a}$ and $T_\beta|_{\mathcal{H}_c}$ to perform the canonical decomposition of $\mathcal{T}_a$ and $\mathcal{T}_c$ respectively. We obtain $2^2=4$ new tuples acting on four orthogonal parts of $\HS$, say, $\HS_{aa}, \HS_{ac}, \HS_{ca}, \HS_{cc}$, respectively, where, $\HS_a=\HS_{aa} \oplus \HS_{ac}$ and $\HS_c=\HS_{ca} \oplus \HS_{cc}$. Let us denote by 
\[
\mathcal{T}_{aa}=\{T_\mu|_{\mathcal{H}_{aa}}  \ : \mu \in \Omega \} \quad \text{and} \quad \mathcal{T}_{ac}=\{T_\mu|_{\mathcal{H}_{ac}}  \ : \mu \in \Omega \}
\]
and 
\[
\mathcal{T}_{ca}=\{T_\mu|_{\mathcal{H}_{ca}}  \ : \mu \in \Omega \} \quad \text{and} \quad \mathcal{T}_{cc}=\{T_\mu|_{\mathcal{H}_{cc}}  \ : \mu \in \Omega \}
\]
respectively. Let
\[
\mathbb{F}_\beta=\left\{(\mathcal{T}_{aa}, \HS_{aa}), \ (\mathcal{T}_{ac}, \HS_{ac}), \ (\mathcal{T}_{ca}, \HS_{ca}), \ (\mathcal{T}_{cc}, \HS_{cc})\right\}.
\]
Continuing this process, we see that after applying the axiom of choice $n$-times, we obtain $2^n$ operator tuples (corresponding to the $2^n$ orthognal splits of $\HS$). Let 
\[
\mathbb{F}_\Omega=\underset{\gamma \in \Omega \cup \{0\}}{\bigcup}\mathbb{F}_\gamma.
\]
One can define a relation $"\leq"$ on $\mathbb{F}_\Omega$ in the following way: for any two members two members $\left(A, \mathcal{H}_{A}\right),\left(B, \mathcal{H}_{B}\right)$ in $\mathbb{F}_\Omega$, with $\left(A, \mathcal{H}_A\right) \leq \left(B, \mathcal{H}_B\right)$ if and only if $\mathcal{H}_{B} \subseteq \mathcal{H}_{A}, \mathcal{H}_{B}$ is a joint reducing subspace for $A$ and $B=A|_{{\mathcal{H}_{B}}}$. It turns out that $(\mathbb{F}_\Omega, \leq)$ is a partially ordered set and $(\mathcal{T}, \HS) \leq (A, \HS_A)$ for any $(A, \HS_A) \in \mathbb{F}_\Omega$.
 We denote by $\mathbb{F}_\infty$ the set of all maximal elements for the \textit{maximal totally ordered sets} (see Definition 6.1 in \cite{Pal}) in $\mathbb{F}_\Omega$. The limiting set $\mathbb{F}_\infty$ must have cardinality $2^{|\Omega|}$, where $|\Omega|$ is the cardinality of the set $\Omega$. Intuitively, this holds because here we consider the collection of all functions $f: \Omega \to \{t_u, t_c\}$ and its cardinality is $2^{|\Omega|}$. Moreover, the set $\mathbb{F}_\infty$ consists of atoms only. An interested reader may refer Section 6 in \cite{Pal} for further details and prove the analogues of the statements there in the setting of $\A_r$-contractions without any difficulties. Using the combinatorial arguments provided in Section 6.2 of the first named author's work \cite{Pal}, we can obtain the desired decomposition for any infinite family of doubly commuting $\A_r$-contractions. This is one of the main results of this article.
\begin{thm}\label{general case}
	Let $\mathcal{T}=\{T_\gamma: \gamma \in \Omega \}$ be an infinite family of doubly commuting $\A_r$-contractions acting on a Hilbert space $\mathcal{H}$ which consists of infinitely many non-atoms. If $S_{\Omega}$ is the set of all functions $f: \Omega \rightarrow\{t_u, t_c\}$ and 
	$
	\mathbb{F}_\infty=\{(A_\alpha, \mathcal{H}_\alpha): \alpha\in S_\Omega\},
	$
	where $A_\alpha=\{T_\gamma|_{\mathcal{H}_\alpha}: \gamma \in \Omega\},$ then we have the following.
	\begin{enumerate}
		\item  $\mathcal{H}$ admits an orthogonal decomposition $\mathcal{H}=\bigoplus_{\alpha \in S_{\Omega}} \mathcal{H}_\alpha$ such that each $\mathcal{H}_\alpha$ is a joint reducing subspace for $\mathcal{T}$.
		\item  For each $\alpha \in S_{\Omega}$, the members of the family $\{T_\gamma|_{\mathcal{H}_\alpha}: \gamma \in \Omega \}$ are all atoms.
		\item  There is exactly one element, say $\left(A_1, \mathcal{H}_1\right)$ in $\mathbb{F}_{\infty}$ such that the members of $A_1$ are all atoms of type $t_u$ and this is determined by the constant function $f_1: \Omega \rightarrow\{t_u, t_c\}$ defined by $f_1(\gamma)=t_u$ for all $\gamma\in \Omega$.
		\item  There is exactly one element, say $\left(A_{2^{|\Omega|}}, \mathcal{H}_{2^{|\Omega|}}\right)$ in $\mathbb{F}_{\infty}$ such that the members of $A_{2^{|\Omega|}}$ are all atoms of type $t_c$ that is  determined by the constant function $f_{2^{|\Omega|}}: \Omega \rightarrow\{t_u, t_c\}$ defined by $f_{2^{|\Omega|}}(\gamma)=t_c$ for all $\gamma \in \Omega$.
		\item  The cardinality of the set $\mathbb{F}_{\infty}$ is $2^{|\Omega|}$, where $|\Omega|$ is the cardinality of the set $\Omega$.
		
	\end{enumerate}
	One or more members of $\left\{\mathcal{H}_\alpha: \alpha \in S_{\Omega}\right\}$ may coincide with the trivial subspace $\{0\}$.
\end{thm}

As a natural consequence of this canonical decomposition, we obtain as a special case the following Wold type decomposition for any infinite family of doubly commuting $\A_r$-isometries.
\begin{thm}\label{Wold_infinite}
	Let $\mathcal{V}=\{V_\gamma: \gamma \in \Omega \}$ be an infinite family of doubly commuting $\A_r$-isometries acting on a Hilbert space $\mathcal{H}$ which consists of infinitely many $\A_r$-isometries that are neither $\A_r$-unitaries nor pure $\A_r$-isometries. If $S_{\Omega}$ is the set of all functions $f: \Omega \rightarrow\{t_u, t_c\}$ and 
$
	\mathbb{F}_\infty=\{(A_\alpha, \mathcal{H}_\alpha): \alpha\in S_\Omega\},
	$
	where $A_\alpha=\{V_\gamma|_{\mathcal{H}_\alpha}: \gamma \in \Omega\},$ then we have the following.
	\begin{enumerate}
		\item  $\mathcal{H}$ admits an orthogonal decomposition $\mathcal{H}=\bigoplus_{\alpha \in S_{\Omega}} \mathcal{H}_\alpha$ such that each $\mathcal{H}_\alpha$ is a joint reducing subspace for $\mathcal{T}$.
		\item  For each $\alpha \in S_{\Omega}$, the family $\{V_\gamma|_{\mathcal{H}_\alpha}: \gamma \in \Omega \}$ consists of $\A_r$-unitaries and $\A_r$-isometries.
		\item  There is exactly one element, say $\left(A_1, \mathcal{H}_1\right)$ in $\mathbb{F}_{\infty}$ such that the members of $A_1$ are all atoms of type $t_u$ which is determined by the constant function $f_1: \Omega \rightarrow\{t_u, t_c\}$ defined by $f_1(\gamma)=t_u$ for all $\gamma \in \Omega$.
		\item  There is exactly one element, say $\left(A_{2^{|\Omega|}}, \mathcal{H}_{2^{|\Omega|}}\right)$ in $\mathbb{F}_{\infty}$ such that the members of $A_{2^{|\Omega|}}$ are all atoms of type $t_c$ and this is determined by the constant function $f_{2^{|\Omega|}}: \Omega \rightarrow\{t_u, t_c\}$ defined by $f_{2^{|\Omega|}}(\gamma)=t_c$ for all $\gamma \in \Omega$.
		\item  The cardinality of the set $\mathbb{F}_{\infty}$ is $2^{|\Omega|}$, where $|\Omega|$ is the cardinality of the set $\Omega$.
		
	\end{enumerate}
	One or more members of $\left\{\mathcal{H}_\alpha: \alpha \in S_{\Omega}\right\}$ may coincide with the trivial subspace $\{0\}$.
\end{thm}
 Needless to mention here, in the statement of the above Theorem, any $\A_r$-isometry which is an atom of type $t_c$ (see Definition \ref{defn:061}) must be a pure $\A_r$-isometry. Thus, Theorem \ref{Wold_infinite} is indeed an extension of Theorem \ref{Wold decomposition} to an infinite family of doubly commuting $\A_r$-isometries. Also, it follows from Theorem \ref{general case} that for a doubly commuting family of $\A_r$-contractions $\{T_\gamma\}_{\gamma \in \Omega}$ acting on a Hilbert space $\HS$, there is a joint reducing closed subspace, say $\HS_1$, such that $\{T_\gamma|_{\HS_1}\}_{\gamma \in \Omega}$ is a commuting family of $\A_r$-unitaries. Thus, it is possible to have an extension of Theorem \ref{thm305} to any infinite family of $\A_r$-unitaries and we present this below. Recall that an $\A_r$-unitary $U$ is a \textit{type u} if $U$ is a unitary and a \textit{type r} if $U$ is an $r$-times unitary.
\begin{thm}
	Let $\mathcal{U}=\{U_\rho : \rho \in \mathcal{P}  \}$ be an infinite family of commuting $\A_r$-unitaries acting on a Hilbert space $\mathcal{H}$. If $S_{\mathcal{P}}$ is the set of all functions $\mu: \mathcal{P}  \rightarrow\{u, r\}$ and 
	$
	\mathbb{F}_\infty=\{(A_\mu, \mathcal{H}_\mu): \mu \in S_\mathcal{P}\},
	$
	where $A_\mu=\{U_\rho|_{\mathcal{H}_\mu}: \rho \in \mathcal{P}\},$ then we have the following.
	\begin{enumerate}
		\item  $\mathcal{H}$ admits an orthogonal decomposition $\mathcal{H}=\bigoplus_{\mu \in S_{\mathcal{P}}} \mathcal{H}_\mu$ such that each $\mathcal{H}_\mu$ is a joint reducing subspace for $\mathcal{T}$.
		\item  $A_\mu$ consists of unitaries and $r$-times unitaries for every $\mu \in S_{\mathcal{P}}$.
		\item  There is exactly one element, say $\left(A_1, \mathcal{H}_1\right)$ in $\mathbb{F}_{\infty}$ such that $A_1$ consists of only unitaries which is determined by the constant function $\mu_1: \mathcal{P} \rightarrow\{u, r\}$ defined by $\mu_1(\rho)=u$ for all $\rho \in \mathcal{P}$.
		\item  There is exactly one element, say $\left(A_{2^{|\mathcal{P}|}}, \mathcal{H}_{2^{|\mathcal{P}|}}\right)$ in $\mathbb{F}_{\infty}$ such that $A_{2^{|\mathcal{P}|}}$ consists of only $r$-times unitaries and this is determined by the constant function $\mu_{2^{|\mathcal{P}|}}: \mathcal{P} \rightarrow\{u, r\}$ defined by $\mu_{2^{|\mathcal{P}|}}(\rho)=r$ for all $\rho \in \mathcal{P}$.
		\item  The cardinality of the set $\mathbb{F}_{\infty}$ is $2^{|\mathcal{P}|}$, where $|\mathcal{P}|$ is the cardinality of the set $\mathcal{P}$.
	\end{enumerate}
	One or more members of $\left\{\mathcal{H}_\mu: \mu \in S_{\mathcal{P}}\right\}$ may coincide with the trivial subspace $\{0\}$.
\end{thm}
\subsection{Decomposition of an infinite family of doubly commuting c.n.u. $\A_r$-contractions}

If we follow the same techniques as mentioned in Sections $6 \ \& \ 7$ of \cite{Pal}, we can easily obtain a canonical decomposition for an infinite family of doubly commuting c.n.u. contractions such that every member orthogonally splits into pure $\A_r$-isometries and c.n.i. $\A_r$-contractions. This is an extension of Theorem \ref{thm605} and an analogue of Theorem \ref{general case} to any infinite family of doubly commuting c.n.u. $\A_r$-contractions $\mathcal{C}=\{T_\eta : \eta \in J \}$. We shall follow the terminologies introduced in Definition \ref{defn:063}. Note that $T$ is a non-fundamental c.n.u. $\A_r$-contraction if it is a c.n.u. $\A_r$-contraction which is neither a pure $\A_r$-isometry nor a c.n.i. $\A_r$-contraction. To avoid triviality, we consider a family of doubly commuting c.n.u. $\A_r$-contractions consisting of infinitely many non-fundamental c.n.u. $\A_r$-contractions. Here we denote the limiting set by $\mathbb{G}_\infty$ what we denoted by $\mathbb{F}_\infty$ in case of doubly commuting $\A_r$-contractions in Subsection \ref{Subsec:final-01}.
\begin{thm}\label{thm805}
	Let $\mathcal{C}=\{T_\eta: \eta \in J \}$ be an infinite family of doubly commuting c.n.u. $\A_r$-contractions acting on a Hilbert space $\mathcal{H}$ consisting of infinitely many non-fundamental c.n.u. $\A_r$-contractions. If $S_J$ is the set of all functions $h: J \rightarrow\{t_p, t_{cni}\}$ and 
	$
	\mathbb{G}_\infty=\{(A_\mu, \mathcal{H}_\mu): \mu\in S_J\},
	$
	where $A_\mu=\{T_\mu|_{\mathcal{H}_\mu}: \mu \in S_J\}$, then the following statements hold.
	
		\begin{enumerate}
		\item  $\mathcal{H}$ admits an orthogonal decomposition $\mathcal{H}=\bigoplus_{\mu \in S_{J}} \mathcal{H}_\mu$ such that each $\mathcal{H}_\mu$ is a joint reducing subspace for $\mathcal{C}$.
		\item  For each $\mu \in S_{J}$, the family $\{T_\eta|_{\mathcal{H}_\mu}: \eta \in J \}$ consists of pure $\A_r$-isometries and c.n.i. $\A_r$-contractions.
		\item  There is exactly one element, say $\left(A_1, \mathcal{H}_1\right)$ in $\mathbb{G}_{\infty}$ such that the members of $A_1$ are all atoms of type $t_p$ and this is determined by the constant function $h_1: J \rightarrow\{t_p, t_{cni}\}$ defined by $h_1(\eta)= t_p$ for all $\eta\in J$.
		\item  There is exactly one element, say $\left(A_{2^{|J|}}, \mathcal{H}_{2^{|J|}}\right)$ in $\mathbb{G}_{\infty}$ such that the members of $A_{2^{|J|}}$ are all atoms of type $t_{cni}$ which is determined by the constant function $h_{2^{|J|}}: J \rightarrow\{t_p, t_{cni}\}$ defined by $h_{2^{|J|}}(\eta)=t_{cni}$ for all $\eta \in J$.
		\item  The cardinality of the set $\mathbb{G}_{\infty}$ is $2^{|J|}$, where $|J|$ is the cardinality of the set $J$.
		
	\end{enumerate}
	One or more members of $\left\{\mathcal{H}_\mu: \mu \in S_J \right\}$ may coincide with the trivial subspace $\{0\}$.
\end{thm}

\vspace{0.4cm}

\noindent \textbf{Concluding remarks.} We conclude this article here and leave with the readers the following questions that we could not answer to in this paper. Whether or not we can find an example of a finite commuting tuple of $\A_r$-contractions that does not admit a canonical decomposition of the kind described in Theorem \ref{thm602}. As we have mentioned in the beginning of Section \ref{doubly commuting} that we should not expect every commuting tuple of $\A_r$-contractions $(T_1, \dots , T_n)$ to be an $\A_r^n$-contraction. However, at this point we do not have a concrete counter example in support of this.

\section{Declarations}

\noindent (1) Data sharing is not applicable to this article as no datasets were generated or analysed during the current study.

\smallskip

\noindent (2) In case any datasets are generated during and/or analysed during the current study, they must be available from the corresponding author on reasonable request.

\smallskip

\noindent (3) There are no competing interests.

\vspace{0.2cm}

\noindent \textbf{Acknowledgement.} The first named author is supported by the `Early Achiever Research Award' of IIT Bombay, India with Grant No. RI/0220-10001427-001. The second named author is supported by the Prime Minister's Research Fellowship with PMRF Id No. 1300140 of Govt. of India.

\end{document}